\documentclass[a4paper,11pt]{amsart}

\newfont{\cyr}{wncyr10 scaled 1100}

\usepackage[usenames,dvipsnames]{color}

\usepackage[left=2.7cm,right=2.7cm,top=3.5cm,bottom=3cm]{geometry}
\usepackage{amsthm,amssymb,amsmath,amsfonts,mathrsfs,amscd,yfonts}
\usepackage{turnstile}
\usepackage[latin1]{inputenc}
\usepackage[all]{xy}
\usepackage{latexsym}
\usepackage{longtable}
\usepackage{hyperref}

\theoremstyle{plain}
\newtheorem{theorem}{Theorem}[section]
\newtheorem{corollary}[theorem]{Corollary}
\newtheorem{lemma}[theorem]{Lemma}
\newtheorem{proposition}[theorem]{Proposition}
\newtheorem{propo}[theorem]{Proposition}
\newtheorem{coro}[theorem]{Corollary}

\theoremstyle{definition}

\newtheorem{examplewr}[theorem]{Example}

\theoremstyle{remark}
\newtheorem{obswr}[theorem]{Observation}
\newtheorem{remarkwr}[theorem]{Remark}

\DeclareMathOperator{\dR}{\mathrm{dR}}

\DeclareMathOperator{\BK}{BK}

\DeclareMathOperator{\et}{et}

\DeclareMathOperator{\cyc}{cyc}
\DeclareMathOperator{\fin}{f}

\DeclareMathOperator{\Eis}{Eis}

\DeclareMathOperator{\dlog}{dlog}

\DeclareMathOperator{\FK}{FK}

\DeclareMathOperator{\sub}{sub}
\DeclareMathOperator{\quo}{quo}

\newcommand{\fM}{{\mathfrak{M}}}

\newcommand{\cW}{\mathcal W}

\newcommand{\Q}{\mathbb{Q}}
\newcommand{\Z}{\mathbb{Z}}

\newcommand{\C}{\mathbb{C}}

\newcommand{\Gal}{\mathrm{Gal\,}}

\newcommand{\Div}{\mathrm{Div}}
\newcommand{\Fil}{\mathrm{Fil}}

\newcommand{\Frob}{\mathrm{Fr}}
\newcommand{\End}{\mathrm{End}}

\newcommand{\Aut}{\mathrm{Aut}}
\newcommand{\Fr}{\mathrm{Fr}}

\newcommand{\ord}{{\mathrm{ord}}}
\newfont{\gotip}{eufb10 at 12pt}

\newcommand{\cO}{{\mathcal O}}

\newcommand{\ra}{\rightarrow}
\newcommand{\lra}{\longrightarrow}

\DeclareMathOperator{\Hom}{Hom}

\newcommand{\res}{\mathrm{res}}

\newcommand{\fp}{{\mathfrak p}}

\include{thebibliography}

\begin{document}

\title[Motivic congruences and Sharifi's conjecture]{Motivic congruences and Sharifi's conjecture}

\author{\'Oscar Rivero and Victor Rotger}

\begin{abstract}
Let $f$ be a cuspidal eigenform of weight two and level $N$, let $p\nmid N$ be a prime at which $f$ is congruent to an Eisenstein series and let $V_f$ denote the $p$-adic Tate module of $f$.
Beilinson \cite{Be} constructed a class $\kappa_f\in H^1(\Q,V_f(1))$ arising from the cup-product of two Siegel units and proved a striking relationship with the first derivative $L'(f,0)$ at the near central point $s=0$ of the $L$-series of $f$, which led him to formulate his celebrated conjecture.  In this note we prove two congruence formulae relating the ``motivic part" of $L'(f,0) \,(\mathrm{mod} \, p)$ and $L''(f,0) \,(\mathrm{mod} \, p)$ with circular units.
The proofs make use of delicate Galois properties satisfied by various integral lattices within $V_f$ and exploits  Perrin-Riou's, Coleman's and Kato's work on  the Euler systems of circular units and Beilinson--Kato elements and, most crucially, the work of Sharifi \cite{Sh}, Fukaya--Kato \cite{FK} and Ohta \cite{Oh1}.
\end{abstract}

\address{O. R.: Mathematics Institute, University of Warwick, Coventry CV4 7AL, United Kingdom}
\email{riverosalgado@gmail.com}

\address{V. R.: IMTech, UPC and Centre de Recerca Matem\`{a}tiques, C. Jordi Girona 1-3, 08034 Barcelona, Spain }
\email{victor.rotger@upc.edu}


\maketitle

\tableofcontents

\section{Introduction}

The aim of this paper is showing how the ideas underlying Sharifi's conjecture \cite{Sh} and the work \cite{FK} of Fukaya and Kato can be exploited to study congruences among motivic classes that naturally arise from the Euler systems of Beilinson-Kato elements and circular units.

In order to set the stage, let $N > 1$ be a positive integer, $\theta: (\Z/N\Z)^\times \ra \bar\Q^\times$ an even Dirichlet character and $f \in S_2(N,\theta)$ a normalized cuspidal eigenform of level $N$, weight $2$ and nebentype $\theta$.

Fix a prime $p\nmid 6N \varphi(N)$. Let $T_{f,X}$ and $T_{f,Y}$  denote the integral $p$-adic Galois representations given as the $f$-isotypical quotient of  $H^1_{\et}(\bar{X},\Z_p(1))$, resp.\,$H^1_{\et}(\bar{Y},\Z_p(1))$ of the closed, resp.\,open modular curve\footnote{Cf.\,\S \ref{sec:back} for the particular models of these curves we employ in this article and precise definitions.
} of level $\Gamma_1(N)$. Set as usual $V_f=T_{f,X}\otimes \Q =T_{f,Y}\otimes \Q$.


Beilinson \cite{Be} introduced a motivic element in the $K$-group $K_2(Y)$ giving rise to a global Galois cohomology class
\begin{equation}
\kappa_f = \kappa_f(\chi_1,\chi_2) \in H^1(\Q, T_{f,Y}(1))
\end{equation}
that was later the basis of Kato's Euler system in \cite{Kato}. This class depends on auxiliary data (cf.\,loc.\,cit.\,and \cite{BD}, \cite{Han}, \cite[\S 9]{KLZ}, \cite{Sch} for several presentations of the subject in the literature). With our normalizations, $\kappa_f$ depends on the choice of two auxiliary Dirichlet characters $\chi_1$ and $\chi_2$ of the same parity (see \S \ref{first-section}) and it is straight-forward to relate it to other equivalent conventions adopted in loc.\,cit.

The main motivation of Beilinson's construction of $\kappa_f$ was providing evidence for his celebrated conjecture on values of $L$-functions of mixed motives, encompassing Dirichlet's unit theorem and the Birch and Swinnerton-Dyer conjecture for elliptic curves as  instances of it. This conjecture was later refined by Bloch-Kato in \cite{BK},  and in the case at hand predicts that
\begin{equation}\label{Bei-conj}
\ord_{s=0} L(f,s) \, \stackrel{?}{=} \dim H^1_{\fin}(\Q,V_f(1)),
\end{equation}
where the left-hand side is the $L$-series associated to $f$ and the right-hand side denotes the space of {\em crystalline} classes in $H^1(\Q,V_f(1))$. Since $ L(f,0)$ vanishes (for innocent reasons, due to the vanishing of a $\Gamma$-factor arising from analytic continuation), \eqref{Bei-conj} suggests that $H^1_{\fin}(\Q,V_f(1))$ should contain non-trivial classes, and Beilinson proved in loc.\,cit.\,that $\kappa_f$ is indeed crystalline and non-trivial, provided $L'(f,0) \ne 0$.

\vspace{0.2cm}

Let $F$ be the finite extension of $\Q$ generated by the field of coefficients of $f$
and the values of all Dirichlet characters of conductor $N$; let $\mathcal O$ be its ring of integers and $\mathfrak p\subset \cO$ a prime ideal above $p$.

In this note we assume $f$ is congruent to an Eisenstein series modulo  $\mathfrak p$. Up to replacing $f$ with a twist of it, we may assume without loss of generality that
\begin{equation}\label{congruence}
f \equiv E_2(\theta,1) \,\, \mathrm{mod} \, \mathfrak p^t
\end{equation}
for some $t \geq 1$, where $E_2(\theta,1)$ is the classical Eisenstein series recalled in \eqref{def-Eis} below.\footnote{Congruence \eqref{congruence} is equivalent to $f^* \equiv E_2(1,\bar \theta)$, where $f^*=f\otimes \theta^{-1}$ is the dual form of $f$, and this in turn implies that $ \mathfrak p^t$ divides the generalized Bernoulli number $B_2(\bar \theta)$ (or equivalently the $L$-value $L(\bar\theta,-1)$).}
We take $t$ to be the largest power satisfying \eqref{congruence}.

As is well-known, $T_{f,X}$ and $T_{f,Y}$ are finitely generated $\cO_{\mathfrak{p}}[G_\Q]$-modules giving rise to the same Galois representation $V_f$ over $F_{\mathfrak{p}}$. However, the lattices $T_{f,X}$ and $T_{f,Y}$ are not isomorphic as $G_{\Q}$-modules, and in general none of them are necessarily free as $\cO_{\mathfrak{p}}$-modules. Setting $ \bar T := T\otimes \mathcal O/\mathfrak p^t$ for any $\cO_{\mathfrak{p}}$-module, congruence \eqref{congruence} does imply (cf.\,\S \ref{sec:back}) that one always has surjective homomorphisms of $G_{\Q}$-modules
\begin{equation}\label{VfX}
\bar{T}_{f,Y}\stackrel{\bar\pi_1}{\lra} \mathcal O/\mathfrak p^t(\theta), \quad  \bar{T}_{f,X} \stackrel{\bar\pi_2}{\lra} \mathcal O/\mathfrak p^t(1).
\end{equation}

The maps $\bar\pi_1$ and $\bar\pi_2$ in \eqref{VfX}  are non-canonical, but we exploit the work of Ohta \cite{Oh1}, \cite{Oh2}, Sharifi \cite{Sh} and Fukaya-Kato \cite{FK} to rigidify them in a canonical way, in the sense that both $\bar\pi_1$ and $\bar\pi_2$ only depend on canonical periods naturally associated to $f$; cf.\,\eqref{rig-iso} and \eqref{rig-iso2}.

Write $\bar \kappa_f = \kappa_{f} \, (\mbox{mod } \mathfrak{p}^t)$ and define
$$
\bar \kappa_{f,1} = \bar\pi_{1*}(\bar\kappa_f)\in H^1(\mathbb Q,\mathcal O/\mathfrak p^t(\theta)(1)).
$$
Motivated by the above discussion of Beilinson's conjecture, $\bar \kappa_{f,1}$ may be regarded as a motivic avatar of the first derivative $L'(f,0)$ $(\mathrm{mod} \, \mathfrak p^t)$. The first main result of this note, Theorem \ref{thm1} below, is an explicit formula for $\bar \kappa_{f,1}$ in terms of algebraic $L$-values and circular units, which in particular provides a criterion for this class to vanish. In the parlance of \cite{BD-Joch}, our Theorem \ref{thm1} may be interpreted as a {\em Jochnowitz congruence}
$$
L_{\mathrm{alg}}'(f,0) \,  \equiv \,  L_{\mathrm{alg}}'(\theta,0) \quad (\mathrm{mod} \, \mathfrak p^t)
$$
between the ``algebraic" or "motivic parts" of the derivative at $s=0$ of the Hasse-Weil $L$-function $L(f,s)$ and Dirichlet's $L$-function $L(\theta,s)$.

We further show that $\kappa_f$ may be lifted  to an element in $H^1(\mathbb Q, T_{f,X}(1))$ if and only if the $(\mathrm{mod} \, \mathfrak p^t)$ class $\bar \kappa_{f,1}$ vanishes. When this happens, such a lift is unique and we thus continue to denote it $\kappa_f$ by slight abuse of notation; we may then define
$$
\bar \kappa_{f,2} = \bar\pi_{2*}(\bar\kappa_{f}) \in H^1(\mathbb Q, \mathcal O/\mathfrak p^t(2)),
$$
which we regard as the motivic counterpart of the second derivative $L^{''}(f,0)$ $(\mathrm{mod} \, \mathfrak p^t)$.

The second main result of this note, Theorem \ref{thm2} below,  provides  an explicit formula for $\bar \kappa_{f,2}$ as the cup-product of two circular units. We find interesting that the circle of ideas appearing in \cite{FK} and \cite{Sh} can be applied to the computation of second derivatives, a type of result which appears to be quite novel.

In order to state our results precisely, let
\begin{equation}\label{units-iso}
\Z[\mu_N]^\times[\theta] = (\Z[\mu_N]^\times \otimes \mathcal O_{\mathfrak p}(\bar \theta))^{\Gal(\Q(\mu_N)/\Q)} \simeq \Hom(\mathcal O_{\mathfrak p}(\theta),\Z[\mu_N]^\times \otimes \mathcal O_{\mathfrak p})
\end{equation}
denote the $\theta$-isotypic component of $\Z[\mu_N]^\times \otimes \mathcal O_{\mathfrak p}$ on which $\Gal(\Q(\mu_N)/\Q)$ acts through $\theta$, which may be naturally identified with a $\mathcal O_{\mathfrak p}$-submodule of $\Z[\mu_N]^\times \otimes \mathcal O_{\mathfrak p}$ of rank 1 when $\theta \neq 1$ (resp. rank 0 when $\theta = 1$).
Kummer theory gives rise to an injective homomorphism
$$\Z[\mu_N]^\times[\bar \theta] \ra \Hom(G_{\Q(\mu_N)}, \mathcal O_{\mathfrak p}(1))[\bar \theta] \ra H^1(\mathbb Q, \mathcal O_{\mathfrak p}(\theta)(1)).$$

Fix a primitive $N$-th root of unity $\zeta_N$ and define the circular unit
\begin{equation}\label{circular}
c_{\theta} := \prod_{a=1}^{N-1}  (1-\zeta_N^a)^{\theta(a)} \in\Z[\mu_N]^\times[\bar \theta].
\end{equation}
Let
$$
c_{\theta} \in H^1(\mathbb Q, \mathcal O_{\mathfrak p}(\theta)(1))
$$
denote, with the same symbol by a slight abuse of notation, its image under the identification provided by the Kummer map. Write $\bar c_{\theta} = c_{\theta} \, (\mbox{mod } \mathfrak{p}^t)\, \in H^1(\mathbb Q, \mathcal O/\mathfrak p^t(\theta)(1))$.

Our main theorems are conditional on the following two hypotheses, that we assume throughout this article:

\begin{enumerate}

\item[(H1)] {\em Non-trivial zeroes mod $\mathfrak{p}$:}
\begin{equation*}
\theta(p)-1, \, \chi_1 \bar \chi_2(p)-1, \, \theta \chi_1 \bar \chi_2(p)-1 \ne 0 \,\, (\mbox{mod} \,\, \mathfrak{p}).
\end{equation*}

\item[(H2)] Letting $\Sigma_X$, $\Sigma_Y$ denote the torsion submodules of $T_{f,X}$ and $T_{f,Y}$ respectively, the $G_{\Q_p}$-module $\mathcal O/\mathfrak p^t(\theta)$ does not show up as a quotient of $\Sigma_Y/\Sigma_X$.

\end{enumerate}

Note that (H1) implies that $\theta$ and in fact $\theta_{|\mathbb Q_p}$ is non-trivial, even mod $\mathfrak{p}$; in particular, $c_{\theta}$ is a non-trivial unit. Note that (H2) follows automatically if the localization of the Hecke algebra acting on $M_2(\Gamma_1(N))$ at the Eisenstein ideal is Gorenstein, as this implies that $T_{f,Y}$ is free as $\mathcal O_{\mathfrak p}$-module. We wonder whether (H2) might be weaker and more tractable than asking the Hecke algebra to be Gorenstein.


Define the algebraic $L$-value $$L^{\mathrm{alg}}(f^*,\bar \chi_1 \chi_2,1) = L(f^*,\bar \chi_1 \chi_2,1)/\Omega_f^+ \in \cO$$ where $\Omega_f^+$ is Shimura's complex period associated to $f^*$, chosen in a specific way that we recall in \S \ref{mazur}. Let also $$\mathfrak g(\chi)=\sum_{a=1}^{N-1} \chi(a)\zeta_N^a$$ denote the Gauss sum attached to a Dirichlet character $\chi$ of conductor $N$.

\begin{theorem}\label{thm1}
In  $H^1(\mathbb Q_p, \mathcal O/\mathfrak p^t(\theta)(1))$ we have
\[ \bar \kappa_{f,1} \equiv \frac{iN}{12} \cdot \frac{\mathfrak g(\theta \chi_1 \bar \chi_2)}{\mathfrak g(\theta)} \cdot L^{\mathrm{alg}}(f^*,\bar \chi_1 \chi_2,1) \times \bar c_{\theta} \pmod{\mathfrak p^t}.\]
This equality takes place globally in $H^1(\mathbb Q, \mathcal O/\mathfrak p^t(\theta)(1))$ if $p$ is $\bar\theta$-regular as specified in \eqref{Gras-hyp}.
\end{theorem}

Let us state now our second main result.
As we describe in more detail in \S \ref{first-section}, Kato's class is constructed as
$$
\kappa_{f} = \pi_{f*}(u \cup v) \in H^1(\Q, T_{f,Y}(1)),
$$
namely the push-forward to the $f$-isotypic component of the cup product of two modular units
$$
u = u_{\chi_1,\chi_2} \quad \mbox{and}  \quad v=u_{\bar \chi_1, \theta \bar \chi_2}
$$
whose logarithmic derivative are respectively the classical Eisenstein series $E_2(\chi_1,\chi_2)$ and $E_2(\bar \chi_1, \theta \bar \chi_2)$ given in  \eqref{def-Eis}.
Note that $u_{\chi_1,\chi_2}$ are determined by their logarithmic derivative only up to a multiplicative constant, and therefore the first non-vanishing coefficient in the Laurent expansion of $u_{\chi_1,\chi_2}$ at $\infty$, which we simply denote $u_{\chi_1,\chi_2}(\infty)$ as in \cite[\S 5]{FK}, may be chosen arbitrarily. Since $\Gal(\Q(\mu_N)/\Q)$ acts on $E_2(\chi_1,\chi_2)$ via $\chi_1$ \cite[Theorem 1.3.1]{St-book}, it is natural to normalize $u_{\chi_1,\chi_2}$ likewise, so that  $u_{\chi_1,\chi_2}(\infty)$ may be any power of the circular unit $c_{\chi_1}$. In the literature one finds different normalizations, typically either $u_{\chi_1,\chi_2}(\infty)=1$ or $c_{\chi_1}$. In the statement below we have chosen to normalize the modular units above so that
$$
u_{\chi_1,\chi_2}(\infty) = c_{\chi_1},  \quad u_{\bar \chi_1, \theta \bar \chi_2}(\infty) = c_{\bar \chi_1}
$$
but any other choice would be perfectly fine, upon replacing accordingly the two circular units appearing in the cup-product below.


\begin{theorem}\label{thm2}
Assume $\theta$ is primitive of conductor $N$. Suppose further that $\bar \kappa_{f,1} = 0$. Let $L_p(\bar \theta,s)$ denote the Kubota--Leopoldt $p$-adic $L$-function attached to $\bar \theta$ and assume $L_p'(\bar \theta,-1)$ is a $p$-adic unit. The following equality holds in $H^1(\mathbb Q,\mathcal O/\mathfrak p^t(2))$: \[ \bar{\kappa}_{f,2} \equiv \frac{L_p'(\bar \theta,-1)}{1-p^{-1}} \cdot  \frac{\bar c_{\bar \chi_1} \cup \bar c_{\chi_1}}{\cup  \log_p(\varepsilon_{\cyc})} \pmod{\mathfrak p^t}. \]
Here $\varepsilon_{\cyc}$ is the cyclotomic character and $1/\cup  \log_p(\varepsilon_{\cyc})$ denotes the inverse of the map \[ H^1(\mathbb Q, \mathcal O/\mathfrak p^t(2)) \rightarrow H^2(\mathbb Q, \mathcal O/\mathfrak p^t(2)), \quad \kappa \mapsto \kappa \cup \log_p(\varepsilon_{\cyc}), \] which is invertible under our assumptions.
\end{theorem}

In the above statement note that our running assumptions imply that $L_p(\bar \theta,-1) \equiv 0 \pmod{\fp^t}$ and it is thus natural that the first derivative of the Kubota--Leopoldt $p$-adic $L$-function makes an appearance.

Theorems \ref{thm1} and \ref{thm2} are proved in \S \ref{first-section} and \S \ref{second-section} respectively. We hope the methods introduced in this note may help to extend Sharifi's conjectures to other scenarios where the theory of Euler systems has experienced exciting progress in recent years (cf.\,e.g.\,\cite{Tale}, \cite{KLZ}, \cite{LPSZ}, \cite{LSZ}).


\vskip 12pt

{\bf Acknowledgements.} We thank P. Wake for guiding us through the delicate subject of integral lattices in Galois representations. We also thank G. Gras for pointing us to suitable references and providing careful explanations of his work. We are also grateful to H. Darmon and R. Sharifi for their comments. 
This project has received funding from the European Research Council (ERC) under the European Union's Horizon 2020 research and innovation programme (grant agreement No 682152). The first author has also received financial support through ``la Caixa" Fellowship Grant for Doctoral Studies (grant LCF/BQ/ES17/11600010). The second author is supported by Icrea through an Icrea Academia Grant.

\section{Eisenstein series, modular curves and lattices}\label{sec:back}

The aim of this section is recalling well-known facts and setting notations concerning Eisenstein series, models of modular curves and various integral lattices associated to them. We  take the chance to prove some elementary relationships among the latter, which are surely well-known to experts but that we include because we failed to find precise references in the literature.

Fix algebraic closures $\bar\Q$, $\bar\Q_p$ of $\Q$ and $\Q_p$ respectively, and  embeddings of $\bar\Q$ into $\bar\Q_p$ and $\C$. The former singles out a prime ideal $\mathfrak p$ of $\cO$ lying above $p$ and we let $\mathcal O_{\mathfrak p}$ denote the completion of $\mathcal O$ at $\mathfrak p$. We also fix throughout an uniformizer $\varpi$ of $\mathcal O_{\mathfrak p}$ and an isomorphism $\mathbb C_p \simeq \mathbb C$.

Given a variety $Y/\Q$ and a field extension $F/\Q$, let $Y_F = Y \times F$ denote the base change of $Y$ to $F$ and set $\bar Y = Y_{\bar\Q}$.
Fix an integer $N\geq 3$ and let $Y_1(N) \subset X_1(N)$ denote the canonical models over $\mathbb Q$ of the (affine and projective, respectively) modular curves classifying pairs $(A,i)$ where $A$ is a (generalized) elliptic curve
and $i: \mu_N \rightarrow A$ is an embedding of group schemes. It is important to recall that this is not the model used by Fukaya and Kato, as they consider the one which classifies pairs $(A,P)$, where $A$ is a (generalized) elliptic curve and $P$ is an $N$-torsion point of it. In any case, the model of \cite{FK} can be obtained from ours just taking the twist by the cocyle \[ \Gal(\mathbb Q(\mu_N)/\mathbb Q) \rightarrow \Aut(Y), \quad s \mapsto \langle s^{-1} \rangle, \] where $\langle s \rangle$ stands for the diamond operator associated to $s \in (\mathbb Z/N\mathbb Z)^{\times}$.

Let $C_N := X_1(N) \setminus Y_1(N)$ denote the finite scheme of cusps; among them one may distinguish the cusp $\infty \in C_N(\Q)$ associated to Tate's elliptic curve over $\Z((q))$, which is rational over $\Q$ in this choice of model (cf.\,e.g.\,\cite[\S 1.3]{St-book}, \cite{St-TAMS}). (Again, note that in the model of \cite[\S1.3.3]{FK}, cusp $\infty$ is not defined  over $\mathbb Q$ but over $\Q(\mu_N)$.)

Assume now that $F$ contains the values of all Dirichlet characters of conductor $N$. Then a basis of $\Eis_2(\Gamma_1(N),F)$ is indexed by triples $(\chi_1,\chi_2,r)$ where $\chi_1$ and $\chi_2$ are primitive Dirichlet characters of conductors $N_1$ and $N_2$ with $N_1 \cdot N_2 \, \mid \, N$, $\chi_1(-1)=\chi_2(-1)$, and $r$ is a positive integer with $1 < rN_1N_2 \mid N$, provided by the Eisenstein series (cf.\,e.g.\,\cite[Theorem 4.6.2]{DS}, \cite[Def.\,3.4.1]{St-book}):
\begin{equation}\label{def-Eis}
E_2(\chi_1,\chi_2,r) = a_0 + \sum_{n=1}^{\infty} \Big( \sum_{d|n} \chi_1(n/d) \chi_2(d) d \Big) q^{rn}, \quad a_0= \begin{cases} \frac{L(\chi_2,-1)}{2} & \mbox{ if } \chi_1=1 \\ 0 & \mbox{ if } \chi_1\ne 1 \end{cases}
\end{equation}
unless $\chi_1=\chi_2=1$, in which case $E_2(1,1,r) = \sum_{n=1}^{\infty} \Big( \sum_{d|n} d \Big) q^{n} - r\sum_{n=1}^{\infty} \Big( \sum_{d|n} d \Big) q^{rn}$.

When $r=1$ we shall simply denote $E_2(\chi_1,\chi_2):= E_2(\chi_1,\chi_2,1)$.

When $\chi_1=1$, the constant term may also be recast as a generalized Bernoulli number: setting $B_2(x) = x^2-x+1/6$, define
$$
B_2(\chi) := N \sum_{a=1}^{N-1} \chi(a) \cdot B_2(a/N)
$$
for any Dirichlet character $\chi$ of conductor $N$. One then has $-2 L(\chi_2,-1) = B_2(\chi_2)$.

Define the group of modular units $U(N)$ as the subgroup of rational functions of $X_1(N)_{\Q(\mu_N)}$ with zeroes and poles concentrated at the cusps, that is to say $$U(N) = \cO(Y_1(N)_{\Q(\mu_N)})^\times.$$
Similarly as in \eqref{units-iso}, let $U(N)[\chi]$ denote the $\chi$-isotypic component of $U(N)\otimes \cO_{\mathfrak{p}}$ on which $\Gal(\Q(\mu_N)/\Q)$ acts through the character $\chi$. In light of \cite[Theorem 1.3.1]{St-book}, there exists a  modular unit  $u_{\chi_1,\chi_2} \in U(N)[\chi_1]$ satisfying
\begin{equation}\label{mod-unit}
\dlog(u_{\chi_1,\chi_2}) = E_2(\chi_1,\chi_2)\frac{dq}{q} \quad \mbox{and} \quad u_{\chi_1,\chi_2}(\infty) = c_{\chi_1}.
\end{equation}

The $q$-expansion of the modular units $u_{\chi_1,\chi_2}$ can be written down explicitly. Given a pair of integers $(a,b)$ between $0$ and $N-1$, not both equal to 0, define the Siegel unit  \[ u_{a,b;N} = q^w \prod_{n \geq 0} \Big( 1-q^{n+a/N} \zeta_N^b \Big) \prod_{n \geq 1} \Big( 1-q^{n-a/N}\zeta_N^{-b} \Big), \] where $w=\frac{1}{12}-\frac{a}{N}+\frac{a^2}{2N^2}$. Then the $q$-expansion of the modular unit $u_{\chi_1,1}$ is given by
\begin{equation}\label{def-siegel}
u_{\chi_1,1} = \frac{-1}{2 \mathfrak g(\bar \chi_1)} \sum_{b=1}^{N-1} \bar \chi_1(b) \otimes u_{0,b;N},
\end{equation}
where here $N$ stands for the conductor of $\chi_1$. Although we will not use them here in this note, similar expressions can be given for $u_{\chi_1,\chi_2}$ for arbitrary $\chi_2$ by averaging $u_{a,b;N}$ and choosing an appropriate uniformizer.

Kummer theory induces a morphism
\begin{equation}
\delta: U(N)[\chi] \rightarrow H_{\et}^1(Y_1(N),\cO_{\mathfrak p}(\bar \chi)(1)).
\end{equation}

Let $\mathbb{T} \subset \End \, H_{\et,c}^1(\overline{Y}_1(N),\mathcal O_{\mathfrak p}(1))$ denote the Hecke algebra acting on the compactly-supported cohomology of the open modular curve generated by the standard Hecke operators  $T_\ell$   for every (good or bad) prime $\ell$ let--commonly denoted $U_\ell$ at primes $\ell \mid N$. Let also $\mathbb{T}^* \subset \End \, H_{\et}^1(\overline{Y}_1(N),\mathcal O_{\mathfrak p}(1))$ denote the Hecke algebra acting on the cohomology of the open modular curve generated by the dual Hecke operators $T_\ell^*$ as defined in \cite[\S 3.4]{Oh1}, \cite[Def. 2.4.3]{KLZ} for every prime $\ell$.

Recall from the introduction the newform $f\in S_2(N,\theta)$ satisfying the congruence $f \equiv E_2(\theta,1) \,\, \mathrm{mod} \, \mathfrak p^t$. Let $f(q)= \sum a_n(f) q^n$ denote its $q$-expansion at the cusp $\infty$. Enlarge $F$ so that it also contains the eigenvalues $\{ a_n(f)\}_{n\geq 1}$, let $\cO$ denote its ring of integers and let $\cO_{\fp}$ denote its completion at the prime ideal $\mathfrak{p}$ over $p$ fixed in the introduction.

Let
$I_f^* = (T^*_\ell-a_\ell(f)) \subset \mathbb{T}^*$ denote the ideal associated to the system of eigenvalues of $f$ with respect to the dual Hecke operators. Define the $\cO_{\fp}$-modules
\begin{equation}\label{defTf}
T_{f,X} = H_{\et}^1(\overline{X}_1(N),\mathcal O_{\mathfrak p}(1))/ I^*_f, \quad T_{f,Y} = H_{\et}^1(\overline{Y}_1(N),\mathcal O_{\mathfrak p}(1))/ I^*_f.
\end{equation}

Note that the two lattices $T_{f,X}$ and  $T_{f,Y}$ may give rise to completely different $\cO_{\fp}[G_{\Q}]$-modules in spite of the fact that the associated rational Galois representations $$V_f := T_{f,X} \otimes F_{\fp} \simeq T_{f,Y} \otimes F_{\fp}$$ are isomorphic.



\begin{proposition}\label{ses-y} The natural inclusion $X \hookrightarrow Y$ induces by push-forward in cohomology a map $T_{f,X} \ra T_{f,Y}$ that sits in an exact sequence of $\cO_{\mathfrak{p}}[G_{\mathbb Q}]$-modules
$$
0 \ra T_{f,X} \lra T_{f,Y} \stackrel{\pi_1}{\lra} \mathcal O/\mathfrak p^t(\theta) \ra 0.
$$
\end{proposition}
\begin{proof}
Write for short $H^1(X) = H_{\et}^1(\bar X_1(N),\cO_{\mathfrak{p}}(1))$ and $H^1(Y) = H_{\et}^1(\bar Y_1(N),\cO_{\mathfrak{p}}(1))$. As it is shown for instance in \cite[\S 1.8]{St-book}, there is a short exact sequence of the form
\begin{equation}\label{xye}
0 \longrightarrow H^1(X) \longrightarrow H^1(Y) \longrightarrow \Div^0[C_N] \longrightarrow 0,
\end{equation}
where $\Div^0[C_N]$ is the free $\cO_{\mathfrak{p}}$-module of degree $0$ divisors supported on $C_N$ with coefficients in $\cO_{\mathfrak{p}}$. Since the map $H^1(X) \hookrightarrow H^1(Y)$ induces an injection \[ H^1(X)/(I_f^* \cdot H^1(Y) \cap H^1(X))  \hookrightarrow H^1(Y)/I_f^* \cdot H^1(Y), \]
the exactness on the left follows once we show that
\begin{equation}
I_f^* \cdot H^1(Y) \cap H^1(X) = I_f^* \cdot H^1(X).
\end{equation}
The inclusion $I_f^* H^1(X) \subset I_f^*  H^1(Y) \cap H^1(X)$ is clear. As for the opposite one, take an element $t \alpha \in I_f^*  H^1(Y) \cap H^1(X)$ with $t \in I_f^*$ and $\alpha \in H^1(Y)$. After inverting $p$, there is an isomorphism \[ H^1(Y)[1/p] \simeq H^1(X)[1/p] \oplus \Div^0[C_N][1/p], \] since \eqref{xye} is split after tensoring with $\mathbb Q_p$ by Manin-Drinfeld's theorem. We may thus write $\alpha = \beta + \gamma$, where $\beta \in H^1(X)[1/p]$ and $\gamma \in \Div^0[C_N][1/p]$. Since $t \alpha \in H^1(X)$, we have  $t  \gamma = 0$, and thus  $t  \alpha = t \beta$ in $I_f^*(H^1(Y) \cap H^1(X)[1/p])$.

Note that $H^1(Y) \cap H^1(X)[1/p] = H^1(X)$. Indeed, otherwise there would exist an element $y \in H^1(Y)\setminus H^1(X)$ such that $p^u y \in H^1(X)$ for some $u$. This would imply that $H^1(Y)/H^1(X)$ contains  non-trivial torsion, but this quotient is isomorphic to $\Div^0[C_N]$, which is free.
It thus follows that $t  \alpha $ lies in $I_f^* H^1(X)$, and this  proves exactness on the left.

\vspace{0.1cm}

In order to conclude the proof of the proposition, let us now show that $T_{f,Y}/T_{f,X} \simeq  \mathcal O/\mathfrak p^t(\theta)$. As it follows from the above, \eqref{xye} induces an exact sequence
\begin{equation}
0 \longrightarrow T_{f,X} \longrightarrow T_{f,Y} \longrightarrow \Div^0[C_N]/I_f^*  \Div^0[C_N] \longrightarrow 0.
\end{equation}

The Hecke action on $\Div^0[C_N]$ is Eisenstein, and the eigenvalues of  $T^*_{\ell}$ are $\chi_1(\ell)+\ell \cdot \chi_2(\ell)$, where $(\chi_1, \chi_2)$ range through pairs of Dirichlet characters of the same parity,  not both trivial,  whose product of conductors is $N$ (as described e.g. in \cite[\S 1.3]{St-book}).

Our running assumptions imply that $E_2(\theta,1)$ occurs with multiplicity one in  $\Div^0[C_N]/\mathfrak p^t$.  Indeed, otherwise there would exist a pair of Dirichlet characters $(\xi_1, \xi_2)$ as above such that $E_2(\theta,1) \equiv E_2(\xi_1,\xi_2) \pmod{\mathfrak p}$. This amounts to saying that $\bar \theta \xi_1 \equiv \xi_2 \equiv 1$ (mod $\mathfrak p^t$), and this is only possible when $(\xi_1,\xi_2) = (\theta,1)$ because $p\nmid \varphi(N)$.

Since $f \equiv E_2(\theta,1)$  (mod $\mathfrak p^t$), it follows that $T_{f,Y}/T_{f,X} \simeq \cO/\mathfrak p^t$ as $\cO_{\mathfrak{p}}$-modules. The action of $G_{\Q}$ on $T_{f,Y}/T_{f,X} $ is given by the character $\theta$ by \cite[Theorem 1.3.1]{St-book}.
\end{proof}

Since $f$ is ordinary at $\mathfrak{p}$, it is well-known (cf.\,e.g.\,\cite[1.7]{FK} for a twisted version that boils down to the one below with our normalizations) that there are exact sequences of finitely generated $\cO_{\mathfrak{p}}[G_{\mathbb Q_p}]$-modules
\begin{eqnarray}\label{filtracio}
0 \rightarrow T_{f,X}^{\sub} \rightarrow T_{f,X} \rightarrow T_{f,X}^{\quo} \rightarrow 0 \\ \nonumber
0 \rightarrow T_{f,Y}^{\sub} \rightarrow T_{f,Y} \rightarrow T_{f,Y}^{\quo} \rightarrow 0
\end{eqnarray}
such that
\begin{enumerate}

\item[(i)] $T_{f,X}^{\quo}$  and $T_{f,Y}^{\quo}$ are unramified as $G_{\mathbb Q_p}$-modules.

\item[(ii)] The map $T_{f,X} \ra T_{f,Y}$ induces an isomorphism $T_{f,X}^{\sub}=T_{f,Y}^{\sub}$ of free $\cO_{\fp}$-modules of rank $1$ on which $G_{\mathbb Q_p}$ acts through the cyclotomic character.

\end{enumerate}

Note that $T_{f,X}^{\quo}$  and $T_{f,Y}^{\quo}$ are finitely generated $\cO_{\mathfrak{p}}$-modules such that $V_{f}^{\quo} := T_{f,X}^{\quo}\otimes F_{\mathfrak{p}} \simeq T_{f,Y}^{\quo}\otimes F_{\mathfrak{p}}$ are $1$-dimensional over $F_{\mathfrak{p}}$. In general though  $T_{f,X}^{\quo}$  and $T_{f,Y}^{\quo}$ are not necessarily free and we let $\Sigma_X$ and $\Sigma_Y$, respectively, their torsion submodules.

Let $\alpha_f\in \cO_{\mathfrak p}^\times$ denote the unit root of the $p$-th Hecke polynomial of $f$.
Let $\psi_f: G_{\Q_p} \lra \cO_{\mathfrak p}^\times$ denote the unramified character characterized by $\psi_f(\Frob_p)=\alpha_f$, so that there is an isomorphism of $F_{\mathfrak p}[G_{\Q_p}]$-modules
\begin{equation}\label{psif}
V_f^{\quo} \simeq F_{\mathfrak p}(\psi_f).
\end{equation}

Given a $\cO_{\mathfrak{p}}$-module $T$, set $\bar T = T \otimes \cO/\mathfrak{p}^t$. For a $G_{\Q}$-module $T$ we let $T^{\pm}$ denote the submodule on which complex conjugation acts as $\pm 1$. As shown by Sharifi in \cite[Theorem 4.3]{Sh} and by Fukaya-Kato in \cite[\S  6.3.1, \S 7.1.11]{FK},
there is an exact sequence of $\cO/\mathfrak{p}^t[G_{\Q}]$-modules
\begin{equation}\label{ses-fk0}
0 \longrightarrow \bar T_{f,X}^+ \longrightarrow \bar T_{f,X} \longrightarrow \bar T_{f,X}^- \longrightarrow 0,
\end{equation}
Note the switch of signs between \eqref{ses-fk0} and \cite[\S 6.3.1, \S 7.1.11]{FK}, which is due to the different Tate twist adopted in the definition of $T_{f,X}$. As explained e.g.\,in \cite[\S 2.5.5]{FKS}
 there are isomorphisms of $\cO/\mathfrak{p}^t[G_{\Q_p}]$-modules
 \begin{equation}\label{isos}
\bar T_{f,X}^+ \simeq \bar T_{f,X}^{\quo}, \quad \bar T_{f,X}^{\sub} \simeq \bar T_{f,X}^-.
\end{equation}
The first isomorphism is given by the composition of the inclusion in \eqref{ses-fk0} with the projection map in \eqref{filtracio} mod $\mathfrak p^t$; the second one is given by the analogous composition obtained by switching the roles of \eqref{ses-fk0} and \eqref{filtracio}. In particular $\bar T_{f,X}^-$ is free of rank $1$ over $\cO/\mathfrak{p}^t$.

As for the lattice associated to the open modular curve, let
\begin{equation}\label{Tf-free}
T_{f,Y,\circ}^{\quo} := T_{f,Y}^{\quo}/\Sigma_Y \simeq \cO_{\mathfrak p}(\psi_f)
\end{equation}
denote the free quotient. The latter isomorphism follows from \eqref{psif}.

Since $f \equiv E_2(\theta,1) \,\mbox{mod}\,\mathfrak p^t$, we have
\begin{equation}\label{Vfcong}
T_{f,Y,\circ}^{\quo} \otimes \mathcal O/\mathfrak p^t \simeq \mathcal O/\mathfrak p^t(\theta),
\end{equation}
which amounts to the congruence $\psi_f \equiv \theta \, (\mbox{mod}\,\mathfrak p^t)$ as unramified characters of $G_{\Q_p}$.

Note that $\pi_1$ factors through  $T_{f,Y}^{\quo}$ because the latter is the maximal unramified $G_{\Q_p}$-quotient of $T_{f,Y}$. Assuming also hypothesis (H2) from the introduction, it follows that $\pi_1$ factors further through the natural projection  $T_{f,Y} \lra T_{f,Y,\circ}^{\quo}$ and may thus be written as
 the following composition of $\cO_{\mathfrak{p}}[G_{\Q_p}]$-modules
\begin{equation}\label{pi1-factor}
\pi_1: \quad T_{f,Y} \longrightarrow T_{f,Y,\circ}^{\quo} \longrightarrow \bar{T}_{f,Y,\circ}^{\quo} \stackrel{\sim}{\lra} \cO/\mathfrak{p}^t(\theta).
\end{equation}
While the first and second arrows are canonical projection maps, the latter isomorphism is non-canonical and depends on the choice of a unit of $\cO/\mathfrak{p}^t$.

\section{First congruence relation}\label{first-section}

Keep the notations and assumptions fixed in the introduction concerning the first congruence relation. Here we shall mainly work with the integral Galois representation associated to the {\em open} modular curve, and hence throughout this section we abbreviate $$T_f := T_{f,Y}.$$

 We begin by recalling more precisely the definition of Kato classes. Choose auxiliary Dirichlet characters $\chi_1$, $\chi_2$ as in the introduction and set $\xi_1 = \bar \chi_1$, $\xi_2 = \theta \bar \chi_2$. By \cite[(1.2)]{Nek}, together with the fact that $H_{\et}^j(\bar V,\mathcal O_{\mathfrak p})$ vanishes for any smooth affine variety $V$ of dimension $d$ and any $j>d$, the Hochschild--Serre's spectral sequence gives rise to an isomorphism
\begin{equation}\label{HS-iso}
H_{\et}^2(Y_1(N), \mathcal O_{\mathfrak p}(2)) \simeq H^1(\mathbb Q, H_{\et}^1(\overline{Y_1(N)}, \mathcal O_{\mathfrak p}(2))).
\end{equation}

In view of \eqref{defTf} there is a $G_{\Q}$-equivariant projection
\[ \pi_f: H_{\et}^1(\overline{Y}_1(N),\mathcal O_{\mathfrak p}(1)) \rightarrow T_f, \]
which in turn gives rise to a homomorphism
\begin{equation}\label{f-isotypic}
\pi_{f*}:  H_{\et}^2(Y_1(N), \mathcal O_{\mathfrak p}(2)) \lra H^1(\mathbb Q, T_{f}(1)).
\end{equation}

It thus makes sense to define
$$
\kappa_{f} := \pi_{f*}(  \delta(u_{\xi_1,\xi_2}) \cup \delta(u_{\chi_1,\chi_2}) )\in H^1(\mathbb Q, T_{f}(1)).
$$

Note that $f$ is ordinary at $p$ because of \eqref{congruence} and label the roots of the $p$-th Hecke polynomial of $f$ as $\alpha_{f}$, $\beta_{f}$ so that $\ord_p(\alpha_f)=0$ and $\ord_p(\beta_f)=1$.

Let $\mathcal W$ denote the weight space defined as the formal spectrum of the Iwasawa algebra $\Lambda=\mathcal O_{\mathfrak p}[[\mathbb Z_p^{\times}]]$. The set of classical points in $\cW$ is given by characters $\nu_{s,\xi}$ of the form $\xi \varepsilon_{\cyc}^{s-1}$ where $\xi$ is a Dirichlet character of $p$-power conductor and $s$ is an integer; this forms a dense subset in $\cW$ for the Zariski topology. Let $\mathcal W^{\circ}$ further denote the set of those points with $\xi=1$; we shall often write $s$ in place of $\nu_s = \nu_{s,1}$. Let $\mathcal W^{\pm} \subset \mathcal W$ denote the topological closure of the set of points $\xi \varepsilon_{\cyc}^{s-1}$ with $(-1)^{s-1} \xi(-1) = \pm 1$. We have $\cW= \cW^+ \sqcup \cW^-$ and we write $\Lambda = \Lambda^+ \oplus \Lambda^-$ for the corresponding decomposition of the Iwasawa algebra.

Let $\psi$ be an even Dirichlet character and let $L_p(f^*,\psi)$ denote the Mazur--Tate--Teitelbaum $p$-adic $L$-function associated to $(f^*,\psi)$.
As discussed in \cite[\S 4.5]{FK},  $L_p(f^*,\psi)$  depends on the choice of two complex periods $(\Omega_f^+, \Omega_f^-)$, which in turn are determined by the choice of generators $\delta^{\pm}$ of $V_f^{\pm} = T_{f,X}^{\pm}\otimes \Q$ as
\begin{equation}\label{Omegas}
\Omega_f^\pm = \int_{\delta^{\pm}} \omega_f.
\end{equation}
The $p$-adic $L$-function $L_p(f^*,\psi)$ is  a rigid-analytic function on $\cW$  characterized by the following  formula interpolating classical $L$-values: let $\xi \, : \, (\mathbb Z/p^n \mathbb Z)^{\times} \longrightarrow \bar{\mathbb Q}^{\times}$ be a homomorphism which does not factor through $(\mathbb Z/p^{n-1}\mathbb Z)^{\times}$. Then
\begin{equation}\label{interpolation}
L_p(f^*,\psi)(\xi \varepsilon_{\cyc}) = \begin{cases}  \mathfrak g(\bar \psi) (1-\psi(p) \alpha_f^{-1})(1-\theta \bar \psi(p) \alpha_f^{-1}) \frac{L(f^*,\psi,1)}{\Omega_f^+} & \mbox{ if } \, \xi=1 \\
 \alpha_f^{-n} \mathfrak g(\bar \psi \bar \xi) \frac{L(f^*,\psi \xi,1)}{\Omega_f^{\pm}}  \mbox{ with } \pm = \xi(-1) & \mbox{ if } \, n\geq 1.
\end{cases}
\end{equation}
The set of characters of the form $\xi \varepsilon_{\cyc}$ as $\xi$ ranges through all Dirichlet characters of arbitrary $p$-power conductor is dense within $\cW$ and hence  $L_p(f^*,\psi)$ is uniquely determined by \eqref{interpolation}. More classically, one can also view $L_p(f^*,\psi)$ as a one-variable $p$-adic $L$-function by setting $L_p(f^*,\psi,s) = L_p(f^*,\psi)(\varepsilon_{\cyc}^s)$.

Define the Euler-like factor
\begin{eqnarray}\label{Eulerf}
\mathcal E_f = (1- \alpha_{f})(1-\beta_{f})(1-\bar \theta \bar \chi_1 \chi_2(p) \beta_{f} p^{-1})(1- \chi_1 \bar \chi_2(p) \beta_{f}  p^{-1}).
\end{eqnarray}
Define also the $p$-adic $L$-value
\[ \ell = -iN \cdot \frac{\mathfrak g(\theta \chi_1 \bar \chi_2)}{\mathfrak g(\theta) \mathfrak g(\chi_1 \bar \chi_2)} \cdot (1-\theta(p)) \cdot \zeta_p(-1) \cdot L_p(f^*,\bar \chi_1 \chi_2,1),
\]

Thanks to Proposition \ref{ses-y}, we can define as in the introduction the class \[ \bar \kappa_{f,1} := \pi_{1*}( \kappa_{f})  = \bar\pi_{1*}( \bar\kappa_{f}) \in H^1(\mathbb Q, \mathcal O/\mathfrak p^t(\theta)(1)). \]

\begin{theorem}\label{unteo} (First congruence relation)
Assuming (H1)-(H2) we have  \[ \mathcal E_{f} \cdot  \bar \kappa_{f,1} \equiv \ell \cdot \bar c_{\theta} \quad \mbox{ in }  \, H^1(\mathbb Q_p, \mathcal O/\mathfrak p^t(\theta)(1)). \]
If $p$ is $\bar\theta$-regular as described in \eqref{Gras-hyp}, this equality takes place in $H^1(\mathbb Q, \mathcal O/\mathfrak p^t(\theta)(1))$.
\end{theorem}

Theorem \ref{thm1} in the introduction readily follows from the above statement. Indeed, it follows from \eqref{interpolation} that \[ L_p(f^*,\bar \chi_1 \chi_2,1) = (1-\bar \theta \bar \chi_1 \chi_2(p)\beta_fp^{-1})(1-\chi_1 \bar \chi_2(p) \beta_fp^{-1}) \times \mathfrak g(\chi_1 \bar \chi_2) \times \frac{L(f^*,\bar \chi_1 \chi_2,1)}{\Omega_f^+}. \]
Since $\zeta_p(-1)=(1-p)\cdot \zeta(-1) =(p-1)\cdot \frac{B_2}{2} = \frac{p-1}{12}$,
the Euler factors in the above theorem cancel out in light of (H1) and the congruence $(\alpha_f, \beta_f) \equiv (\theta(p),p) \pmod{\mathfrak p^t}$. Then, Theorem \ref{thm1} follows.

The remainder of this section is devoted to the proof of Theorem \ref{unteo}.


\subsection{Mazur's factorization formula}\label{mazur}

The aim of this section is to recall a mod $\mathfrak{p}$ factorization formula for $L_p(f^*,\psi)$ in terms of two Kubota--Leopoldt $p$-adic $L$-functions. The first result in this direction was proved  by Mazur in \cite{Maz}, and further generalizations were provided by Stevens \cite{St-book} and Greenberg--Vatsal \cite[Theorem 3.12]{GV}. As explained before, this formula actually depends on the choice of periods $\Omega_f^\pm$ and we make this choice by invoking the work of Sharifi  \cite{Sh} and Fukaya-Kato \cite{FK}, which gives rise to a rather precise formula that turns out to be crucial for our purposes.

Note firstly that with the conventions adopted above, character $\xi \varepsilon_{\cyc}$ lies in $\cW^+$ if and only if $\xi$ is odd, and in that case the period appearing in \eqref{interpolation} is $\Omega_f^-$. The main result we aim to prove in this section is concerned with the value of $L_p(f^*,\psi,s)$ at $s=2$ (and $\xi=1$), which again lies in $\cW^+$. For this reason it will suffice to work with the restriction of $L_p(f^*,\psi)$ to $\cW^+$, and accordingly  we only need to choose period $\Omega_f^-$.

The work of Sharifi \cite{Sh} and Fukaya-Kato \cite{FK} allows to take a canonical choice of $\Omega_f^-$. Namely, recall from \eqref{isos} that $T_{f,X}^-$ is a free $\cO_{\mathfrak{p}}$-module of rank $1$. Take any generator $\delta^-$ of $T_{f,X}^-$ such that $\delta^-$ (mod $\mathfrak p^t$) is the basis of $\bar{T}_{f,X}^-$ specified in \cite[6.3.18, 7.1.11, 8.2.4]{FK}, which in turn builds on the work of Sharifi \cite[\S 4]{Sh} and Ohta \cite{Oh1}, \cite{Oh2}. Define $\Omega_f^-$ from $\delta^-$ as in \eqref{Omegas}.

It is readily verified that this choice of $\Omega_f^-$ satisfies the defining properties imposed by Vatsal in \cite{Va}. Beware however that in loc.\,cit.\,these periods are only well-defined up to $p$-adic units and  this is not enough for our aims in this paper.

With our canonical definition of $\Omega_f^-$ at hand, we can now define $\Omega_f^+$ as in \cite[Remark 2.7]{Va2}, namely the one satisfying
\begin{equation}\label{prod-peri}
\Omega_f^+ \Omega_f^- =  \frac{4\pi^2}{\varpi^r} \langle f,f \rangle.
\end{equation}


Given an even Dirichlet character $\chi$, let $L_p(\chi) \in \Lambda$ denote the associated Kubota-Leopoldt $p$-adic $L$-function, which again we may regard as a one-variable function by setting $L_p(\chi,s)=L_p(\chi)(\varepsilon_{\cyc}^s)$.

\begin{propo}\label{fact-l}
The following congruence relation holds in $\Lambda^+$:
\begin{equation}
L_p(f^*,\psi,s) \equiv 2 \cdot L_p(\bar \psi,1-s) \cdot L_p(\bar \theta \psi, s-1) \pmod{\mathfrak p^t}.
\end{equation}
\end{propo}

\begin{proof}
This follows from \cite[Proposition 8.2.4]{FK}.
\end{proof}


In particular, since $\psi$ is even, it holds that
\begin{equation}
L_p(f^*,\psi,2) = 2 \cdot L_p(\bar \psi,-1) \cdot L_p(\bar \theta \psi,1) \equiv -B_2(\bar \psi) \cdot L_p(\bar \theta \psi,1) \pmod{\mathfrak p^t}.
\end{equation}

Let $C_{f^*} \subseteq O_{\mathfrak p}$ denote the congruence ideal attached to $f^*$ as defined e.g.\,in \cite{Oh3}. In Hida's terminology, the appropriate power of the uniformizer $\varpi$ generating $C_{f^*}$, say $\varpi^r$, is sometimes called a congruence divisor.



\subsection{Dieudonn\'e modules and congruences among Ohta's periods}\label{Ohta}

Given a $p$-adic de Rham representation $V$ of $G_{\Q_p}$ with coefficients in $F_{\mathfrak p}$, let $D_{\dR}(V) = (V\otimes_{\Q_p} B_{\dR})^{G_{\Q_p}}$ denote its de Rham Dieudonn\'e module and let $\log_{\BK}$ (resp. $\exp_{\BK}^*$) stand for the Bloch--Kato logarithm (resp. dual exponential map) attached to $V$ as defined in \cite{BK}, \cite{Bel}.

In this note we shall work with the Dieudonn\'e modules associated to the following two basic types of representations. Firstly, if $\chi: G_{\Q_p} \lra \cO^\times_{\mathfrak p}$ is a finite order character and $s$ is an integer, let $F_{\mathfrak p}(\chi)(s)$ denote the $1$-dimensional representation on which $G_{\Q_p}$ acts via $\chi \varepsilon_{\cyc}^s$. Then $D_{\dR}(F_{\mathfrak p}(\chi)(s)$ is again $1$-dimensional and  a canonical generator of $D_{\dR}(F_{\mathfrak p}(\bar \chi)(-s))$ is given by $t^s \mathfrak g(\chi)^{-1}$, where $t$ is Fontaine's $p$-adic analogue of $2 \pi i$. Moreover, there is a perfect pairing
$$\langle \,, \rangle \, : \, D_{\dR}(F_{\mathfrak p}(\chi)(s))\times D_{\dR}(F_{\mathfrak p}(\bar \chi)(-s)) \lra F_{\mathfrak p}$$
which gives rise to the isomorphism
\begin{equation}\label{dtz}
D_{\dR}(F_{\mathfrak p}(\chi)(s)) \rightarrow F_{\mathfrak p} \quad c \mapsto \langle c, \frac{t^s}{\mathfrak g(\chi)} \rangle.
\end{equation}

If $s\geq 1$ and $\chi\ne 1$  when $s=1$, Bloch-Kato's  logarithm  gives rise to an isomorphism
\begin{equation}\label{htd}
\log_{\BK}: H^1(\mathbb Q_p, F_{\mathfrak p}(\chi)(s)) \rightarrow D_{\dR}(F_{\mathfrak p}(\chi)(s)).
\end{equation}

Secondly, the de Rham Dieudonn\'e module $D_{\dR}(V_f)$ associated to the eigenform $f$ is a $F_{\mathfrak p}$-filtered vector space of dimension $2$.
As discussed in e.g.\,\cite[\S 2.8]{KLZ}, Poincar\'e duality yields a perfect pairing \[ \langle \, , \, \rangle: D_{\dR}(V_f(-1)) \times D_{\dR}(V_{f^*}) \rightarrow F_{\mathfrak p} \] and \eqref{filtracio} induces an exact sequence of Dieudonn\'e modules
\begin{equation}\label{Ds}
0 \rightarrow D_{\dR}(V_f^{\sub}) \rightarrow D_{\dR}(V_f) \rightarrow D_{\dR}(V_f^{\quo}) \rightarrow 0,
\end{equation}
where $D_{\dR}(V_f^{\sub})$ and $D_{\dR}(V_f^{\quo})$ have both dimension $1$.

Falting's theorem associates to $f$ a regular differential form $\omega_f \in \Fil(D_{\dR}(V_f))$, which gives rise to an element in $D_{\dR}(V_f^{\quo})$ via the right-most map in \eqref{Ds} and in turn induces a linear form
\begin{equation}\label{omegaf}
\omega_f \, : \, D_{\dR}(V_{f^*}^{\sub}(-1)) \rightarrow F_{\mathfrak p}, \quad \eta \mapsto \langle \omega_f, \eta \rangle
\end{equation}
that we continue to denote with same symbol by a slight abuse of notation.

There is also a differential $\eta_f$, which is characterized by the property that it spans the line $D_{\dR}(V_f^{\sub}(-1))$ and is normalized so that
\begin{equation}\label{def-etaf}
\langle \eta_f, \omega_{f^*} \rangle = 1.
\end{equation}
Again, it induces a linear form
\begin{equation}\label{etaf0}
\eta_f \, : \, D_{\dR}(V_{f^*}^{\quo}) \rightarrow F_{\mathfrak p}, \quad \omega \mapsto \langle \eta_f, \omega \rangle
\end{equation}

\vspace{0.2cm}

We turn now to the more delicate $p$-adic Hodge theory of {\em integral} Galois representations. Let $T$ be an unramified $\cO_{\mathfrak p}[\Gal(\bar\Q_p/\Q_p)]$-module and set $V=T\otimes F_{\mathfrak p}$. Let $\hat\Z_p^\mathrm{ur}$ denote the completion of the ring of integers of the maximal unramified extension of $\Q_p$, and define as in e.g.\,\cite[Theorem 2.1.11]{Oh2} the integral Dieudonn\'e module
$$
D(T) := (T \,  \hat\otimes_{\Z_p}  \, \hat\Z_p^\mathrm{ur})^{\Fr_p=1}.
$$
As shown in loc.\,cit.\,we have $D_{\dR}(V) = D(T)\otimes F_{\mathfrak p}$.

As explained e.g.\,in \cite[Prop.\,1.7.6]{FK}, there is a functorial isomorphism of $\cO_{\mathfrak p}$-modules (forgetting the Galois structure) given by
\begin{equation}\label{D(T)=T}
T  \, \stackrel{\sim}{\lra } D(T).
\end{equation}
This map is not canonical as it depends on a choice of root of unity; for $T=\cO_{\mathfrak p}(\chi)$ we take it to be given by the rule $1 \mapsto \, \mathfrak g(\chi)$.

Recall from \eqref{Tf-free} the free $\cO_{\mathfrak{p}}$-quotient
$$
T_{f,\circ}^{\quo}  \simeq \cO_{\mathfrak p}(\psi_f);
$$
note in particular that $T_{f,\circ}^{\quo}$ is unramified.

According to \cite[\S 10.1.2]{KLZ}, the image of $\eta_{f^*|D(T_{f}^{\quo})}$ is precisely the inverse $C_{f^*}^{-1}$ of the congruence ideal introduced above, and hence there is an isomorphism
\[ \eta_{f^*} \, : \, D(T_{f,\circ}^{\quo}) \longrightarrow C_{f^*}^{-1}, \quad \omega \mapsto \langle \eta_{f^*},\omega \rangle. \]

Writing $C_f=(\varpi^r)$ for the appropriate power  $r\geq 1$ and setting $\tilde \eta_{f^*} := \varpi^r \cdot \eta_{f^*}$, the above map gives rise to an isomorphism of $O_{\mathfrak p}$-modules
\begin{equation}\label{etaf}
\tilde \eta_{f^*} \, : \, D(T_{f,\circ}^{\quo}) \longrightarrow \mathcal O_{\mathfrak p}, \quad \omega \mapsto \langle \tilde \eta_{f^*},\omega \rangle.
\end{equation}

\vspace{0.3cm}


The map $\pi_1$ appearing in Proposition \ref{ses-y} \[ \pi_1 \, : \, T_f \longrightarrow \mathcal O/\mathfrak p^t(\theta) \] is only well-defined up to units in $\mathcal O/\mathfrak p^t$. We rigidify it by invoking diagram \eqref{pi1-factor}, which tells us that $\pi_1$ is fixed once we take a choice of an isomorphism of $G_{\Q_p}$-modules
\begin{equation}\label{bar-iota}
\bar\iota: \bar{T}_{f,\circ}^{\quo} \stackrel{\sim}{\lra} \cO/\mathfrak{p}^t(\theta).
\end{equation}
Fixing such a map amounts to choosing the class $(\mathrm{mod}\,{\mathfrak p^t})$ of an isomorphism of local modules $\iota: T_{f,\circ}^{\quo}  \simeq \mathcal O_{\mathfrak p}(\psi_f)$.
In light of the functoriality provided by \eqref{D(T)=T} this determines and is determined by the class $(\mathrm{mod}\,{\mathfrak p^t})$ of an isomorphism $D(\iota): D(T_{f,\circ}^{\quo}) \simeq D(\cO_{\mathfrak p}(\psi_f))$.


We choose $\iota$ as the single isomorphism  making the following diagram commutative:
\begin{equation}\label{rig-iso}
\xymatrix
{ D(T_{f,\circ}^{\quo})  \quad  \ar[r]^{\quad \langle \, , \tilde \eta_{f^*}\rangle}\ar[d]^{D(\iota)} & \quad  \mathcal O_{\mathfrak p}  \\
 D(\cO_{\mathfrak p}(\psi_f))  \ar[ur]_{\cdot 1/ \mathfrak g(\theta)}&
 }
\end{equation}
Indeed, since both $ \langle \, , \tilde \eta_{f^*}\rangle$ and $\cdot 1/ \mathfrak g(\theta)$ are isomorphisms, it follows that such a map $D(\iota)$ exists and is unique, and this in turn pins down $\iota$ and $\bar\iota$ in light of \eqref{D(T)=T}.

\subsection{Coleman's power series and the Kubota-Leopoldt $p$-adic $L$-function}

Let $$\underline{\varepsilon}_{\cyc}: G_{\mathbb Q} \rightarrow \Lambda^{\times}$$ denote the $\Lambda$-adic cyclotomic character which sends a Galois element $\sigma$ to the group-like element $[\varepsilon_{\cyc}(\sigma)]$. It interpolates the powers of the $\mathbb Z_p$-cyclotomic character, in the sense that for any arithmetic point $\nu_{s,\xi}\in \cW^{\mathrm{cl}}$, \begin{equation}\label{lambda-eps-cyc} \nu_{s,\xi} \circ \underline{\varepsilon}_{\cyc} = \xi \cdot \varepsilon_{\cyc}^{s-1}. \end{equation}

The following result follows from the general theory of Perrin-Riou maps (see for instance \cite[\S 8]{KLZ}).

\begin{propo}\label{perrin-circ}
There exists a morphism of $\Lambda$-modules
\[
\mathcal L_{\chi}: H^1(\mathbb Q_p, \cO_{\mathfrak p}(\chi) \otimes \Lambda(\varepsilon_{\cyc} \underline{\varepsilon}_{\cyc})) \rightarrow \Lambda
\]
satisfying that for all integers $r$, the specialization of $\mathcal L_{\chi}$ at $s \in \mathcal W^{\circ}$ is the homomorphism
\[
\mathcal L_{\chi,s}: H^1(\mathbb Q_p,\mathcal O_{\mathfrak p}(\chi)(s)) \rightarrow \mathcal O_{\mathfrak p}
\]
given by
\[
\mathcal L_{\chi,s} = \frac{1-\bar \chi(p)p^{-s}}{1-\chi(p)p^{s-1}} \cdot
\begin{cases}
\frac{(-1)^s}{(s-1)!} \cdot \langle \log_{\BK},  \frac{t^s}{\mathfrak g(\chi)} \rangle  & \text{ if } s \geq 1 \\
(-s)! \cdot \langle \exp_{\BK}^*, \frac{t^s}{\mathfrak g(\chi)} \rangle  & \text{ if } s<1,
\end{cases}
\]
\end{propo}

As a piece of notation, and for any $p$-adic representation $V$, we write $H_{\fin}^1(\mathbb Q,V)$ for the {\it finite} Bloch--Kato Selmer group, which is the subspace of $H^1(\mathbb Q,V)$ which consists on those classes which are crystalline at $p$ and unramified at $\ell \neq p$.

The following result is a reformulation of Coleman and Perrin-Riou's reciprocity law (\cite{Col}, \cite{PR}), with the normalizations used for instance in \cite{Tale}.

\begin{propo}\label{Leop}
There exists a $\Lambda$-adic cohomology class \[ \kappa_{\chi,\infty} \in H^1(\mathbb Q, \mathcal O_{\mathfrak p}(\chi) \otimes \Lambda(\varepsilon_{\cyc} \underline{\varepsilon}_{\cyc})) \] such that:
\begin{enumerate}
\item[(a)] Its image under restriction at $p$ followed by the Perrin--Riou regulator gives the Kubota--Leopoldt $p$-adic $L$-function: \[ \mathcal L_{\chi}(\res_p(\kappa_{\chi,\infty})) = L_p(\bar \chi). \]
\item[(b)] The bottom layer $\kappa_{\chi}(1) := \nu_1(\kappa_{\chi,\infty})$ lies in $H_{\fin}^1(\mathbb Q, \mathcal O_{\mathfrak p}(\chi)(1))$ and satisfies
\[ \kappa_{\chi}(1) = (1-\chi(p)) \cdot c_{\chi}. \]

\end{enumerate}
\end{propo}

\subsection{Kato's explicit reciprocity law}

\begin{propo}\label{perrin}
There exists a homomorphism of $\Lambda$-modules
\[
\mathcal L_{f}^-: H^1(\mathbb Q_p, T_{f,\circ}^{\quo} \otimes \Lambda(\varepsilon_{\cyc} \underline{\varepsilon}_{\cyc})) \rightarrow \Lambda
\]
satisfying the following interpolation property: for $s \in \mathcal W^{\circ}$, the specialization of $\mathcal L_{f}^-$ at $s$ is the homomorphism
\[
\mathcal L_{f,s}^-: H^1(\mathbb Q_p, T_{f,\circ}^{\quo}(s)) \longrightarrow \mathcal O_{\mathfrak p}
\]
given by
\[
\mathcal L_{f,s}^- = \frac{1- \bar \theta(p) \beta_{f} p^{-s-1}}{1-\theta(p)\beta_f^{-1} p^s} \times \begin{cases} \frac{(-1)^s}{(s-1)!} \times \langle \log_{\BK}, t^s \tilde \eta_{f^*} \rangle & \text{ if } s \geq 1 \\ (-s)! \times \langle \exp_{\BK}^*, t^s \tilde \eta_{f^*} \rangle & \text{ if } s < 1, \end{cases}
\]
where $\log_{\BK}$ is the Bloch-Kato logarithm and $\exp_{\BK}^*$, the dual exponential map.
\end{propo}
\begin{proof}
This follows from Coleman and Perrin-Riou's theory of $\Lambda$-adic logarithm maps as extended by Loeffler and Zerbes in \cite{LZ}. This is recalled for instance in \cite[\S 8,9]{KLZ}. More precisely, \cite[Theorem 8.2.3]{KLZ} and, more particularly, the second displayed equation in  \cite[p.\,82]{KLZ} yields an injective map
$$
H^1(\mathbb Q_p, T_{f,\circ}^{\quo} \otimes \Lambda(\varepsilon_{\cyc} \underline{\varepsilon}_{\cyc})) \lra D(T_f^{\quo}) \otimes \Lambda,
$$
since $H^0(\mathbb Q_p,T_{f,\circ}^{\quo}(1)) = 0$ because of the assumption that $\alpha_f \equiv \theta(p) \neq 1$ modulo $\mathfrak p$.

This map is characterized by the interpolation property formulated in \cite[Appendix B]{LZ}.
Next we apply the pairing of \eqref{etaf} and the result follows.
\end{proof}

\begin{theorem}\label{kato-law}
There exists a $\Lambda$-adic cohomology class \[ \kappa_{f,\infty} \in H^1(\mathbb Q, T_f \otimes \Lambda^-(\varepsilon_{\cyc} \underline{\varepsilon}_{\cyc})) \] such that:
\begin{enumerate}

\item[(a)] There is an explicit reciprocity law
\[ \mathcal L_{f}^-(\res_p(\kappa_{f,\infty})^-) = \frac{N \mathfrak g(\theta\chi_1 \bar \chi_2) \cdot L_p(f^*, \bar \chi_1 \chi_2,1)}{2i \mathfrak g(\chi_1 \bar \chi_2)} \times L_p(f^*,1+s), \]
where $\res_p$ stands for the map corresponding to localization at $p$ and $\res_p(\kappa_{f,\infty})^-$ is the map induced in cohomology from the projection map $T_f \rightarrow T_{f,\circ}^{\quo}$ of \eqref{filtracio}.

\item[(b)]  The bottom layer $\kappa_{f}(1)$ lies in $H_{\fin}^1(\mathbb Q, T_{f}(1))$ and satisfies \[ \kappa_{f}(1) = \mathcal E_f \cdot \kappa_{f}, \] where $\mathcal E_f$ is the Euler factor introduced in \eqref{Eulerf}.
\end{enumerate}
\end{theorem}

\begin{proof} This is due to Kato \cite{Kato} and has been reported in many other places in the literature. See \cite{Och} and, more specifically, \cite[Theorems 4.4 and 5.1]{BD} combined with Besser's  \cite[Proposition 9.11 and Corollary 9.10]{Bes} showing that the $p$-adic regulator can be recast as the composition of the $p$-adic \'etale regulator followed by the Bloch--Kato logarithm.

Note however that the normalizations adopted in loc.\,cit.\,are slightly different from ours. More precisely, in \cite{BD} employ the functional $\eta_f$ instead of $\tilde\eta_f$, but note that this discrepancy is compensated by our choice of periods $\Omega_f^+$ and $\Omega_f^-$,  which we have normalized according to \eqref{prod-peri}. Taking this into account, the theorem holds as stated.
\end{proof}

Recall our running assumption that $f\equiv E_2(\theta,1) \, \pmod{\mathfrak p^t}$.

\begin{corollary}\label{qualsevol}
The following equality holds in $\Lambda^-/\mathfrak p^t \Lambda^-$: \[ \mathcal L_f^-(\res_p \, \kappa_{f,\infty}^-) \equiv (-i)N \frac{\mathfrak g(\theta \chi_1 \bar \chi_2)}{\mathfrak g(\chi_1 \bar \chi_2)} \zeta_p(-1) \cdot L_p(f^*,\bar \chi_1 \chi_2,1) \cdot \mathcal L_{\theta}(\res_p\, \kappa_{\theta,\infty}) \pmod{\mathfrak p^t}. \]
\end{corollary}
\begin{proof}
By Proposition \ref{fact-l}, \[  L_p(f^*,1+s) \equiv 2 \cdot \zeta_p(-s) \cdot L_p(\bar \theta,s) \pmod{\mathfrak p^t}. \] Applying now part (a) of Theorem \ref{kato-law} and Proposition \ref{Leop} to the left and right hand sides respectively, the result follows.
\end{proof}

\subsection{Proof of Theorem \ref{unteo}}

We now prove Theorem \ref{unteo}. Note that since $\theta_{|G_{\Q_p}}$ is a non-trivial unramified character, it follows from e.g.\,\cite[\S2.2]{Bel} that \[ H^1(\mathbb Q_p, \mathcal O_{\mathfrak p}(\theta)(1)) = H_{\fin}^1(\mathbb Q_p, \mathcal O_{\mathfrak p}(\theta)(1)). \]  We have seen in Proposition \ref{perrin-circ} that there is a homomorphism
\begin{equation}\label{isozp}
\mathcal L_{\theta,1}: H^1(\mathbb Q_p, \mathcal O_{\mathfrak p}(\theta)(1)) \rightarrow L_{\mathfrak p}.
\end{equation}

\begin{lemma}\label{l-theta-iso}
The map $\mathcal L_{\theta,1}$ induces an isomorphism \[ \mathcal L_{\theta,1} : H^1(\mathbb Q_p, \mathcal O_{\mathfrak p}(\theta)(1)) \rightarrow \mathcal O_{\mathfrak p}. \]
\end{lemma}

\begin{proof} According to Proposition \ref{perrin-circ}, the map \eqref{isozp} is given by
$$
\mathcal L_{\theta,1}=  \frac{\bar \theta(p)p^{-1}-1}{1-\theta(p)} \cdot \langle \log_{\BK},  \frac{t}{\mathfrak g(\theta)}\rangle.
$$

Given a place $v$ of $\mathbb Q(\mu_N)$ above $p$, let $\mathbb Z[\mu_N]_v$ denote the completion of $\mathbb Z[\mu_N]$ at $v$. Define the module of local units $U_p(N) = \prod_{v | p} \mathbb Z[\mu_N]_v^{\times}$, where $v=v_1,\ldots,v_r$ ranges over all places of $\mathbb Q(\mu_N)$ above $p$. Note that $G = \Gal(\mathbb Q(\mu_N)/\mathbb Q)$ acts on $U_p(N)$ by permuting the places $v$, and hence it makes sense to pick the eigen-component of $U_p(N)$ with respect to a character of $G$. In particular, we have \[ U_p(N)[\bar \theta] := (U_p(N) \otimes \mathcal O_{\mathfrak p}(\theta))^G. \]

Kummer theory identifies $H^1(\mathbb Q_p,\mathcal O_{\mathfrak p}(\theta)(1))$ with $U_p(N)[\bar \theta]$, which is a $\mathcal O_{\mathfrak p}$-module of rank one.

Since $\mathbb Q(\mu_N)_v$ is an unramified extension of $\mathbb Q_p$, the maximal ideal of $\mathbb Z[\mu_N]_v$ is $p \mathbb Z[\mu_N]_v$ and the logarithm defines an isomorphism, as recalled for instance in \cite[\S8]{Con}
\begin{equation}\label{logv}
\log_v \, : \, \mathbb Z[\mu_N]_v^{\times} \otimes \mathcal O_{\mathfrak p} \longrightarrow p \mathbb Z[\mu_N]_v \otimes \mathcal O_{\mathfrak p}.
\end{equation}
Note that $\prod_v \mathbb Z[\mu_N]_v$ is naturally a $G$-module isomorphic to the regular representation and hence $(\prod_v \mathbb Z[\mu_N]_v)[\bar \theta]$ is again a free module of rank $1$ over $\mathcal O_{\mathfrak p}$. Define
\[
\log_{\bar \theta} := \sum_{\sigma \in G} \theta(\sigma) \log_{\sigma(v_1)} \, : \, U_p(N)[\bar \theta] \longrightarrow p (\prod_v \mathbb Z[\mu_N]_v)[\bar \theta].
\]
A natural generator of the target may be taken to be the Gauss sum $\mathfrak g(\theta)$ diagonally embedded in $\prod_v \mathbb Z[\mu_N]_v$ and this yields an identification $\frac{1}{\mathfrak g(\theta)}(\prod_v \mathbb Z[\mu_N]_v)[\bar \theta] = \mathcal O_{\mathfrak p}$. Under these identifications, Bloch-Kato's logarithm may be recast classically as
\[
\langle \log_{\BK}, \frac{t}{\mathfrak g(\theta)} \rangle = \frac{1}{\mathfrak g(\theta)} \log_{\bar \theta}: H^1(\mathbb Q_p,\mathcal O_{\mathfrak p}(\theta)(1)) = U_p(N)[\bar \theta] \lra \mathcal O_{\mathfrak p},
\]
and we already argued that this yields an isomorphism onto $p \mathcal O_{\mathfrak p}$.

Since $\ord_p (\frac{\bar \theta(p)p^{-1}-1}{1-\theta(p)}) = -1$, it follows that $\mathcal L_{\theta,1}$ is an isomorphism onto $\mathcal O_{\mathfrak p}$, as claimed.
\end{proof}

Recall from Proposition \ref{perrin} the map
$$
\mathcal L_{f,1}^-: H^1(\mathbb Q_p, T_{f,\circ}^{\quo}(1)) \, \rightarrow D(T_{f,\circ}^{\quo}(1)) \, \stackrel{\cdot t \tilde\eta_{f^*}}{\rightarrow} \,\mathcal O_{\mathfrak p}.
$$
Recall the isomorphism $T_{f}^{\quo} \otimes \mathcal O/\mathfrak p^t \simeq  \mathcal O/\mathfrak p^t(\theta)$ of \eqref{bar-iota} fixed as in \eqref{rig-iso} above and use it to identify the source of $\mathcal L_{f,1}^- \otimes \cO/\mathfrak{p}^t$ with $H^1(\mathbb Q_p, \mathcal O/\mathfrak{p}^t(\theta)(1))$.

\begin{lemma}\label{pr-iguals}
As homomorphisms $H^1(\mathbb Q_p, \mathcal O/\mathfrak{p}^t(\theta)(1)) \lra \cO/\mathfrak{p}^t$ we have the congruence
\[
\mathcal L_{f,1}^- \equiv \mathcal L_{\theta,1} \pmod{\mathfrak p^t}.
\]
\end{lemma}

\begin{proof}
This follows by comparing the maps $\mathcal L_{\theta,1}$ and $\mathcal L_{f,1}^-$ described respectively in Proposition \ref{perrin-circ} and \ref{perrin}. Note firstly that the Euler factors involved in the latter agree modulo $\mathfrak p^t$ with those of the former, since $\alpha_f \beta_f = \theta(p)p$.

Next, observe that in Proposition \ref{perrin-circ} the pairing takes place against $t\mathfrak g(\chi)^{-1}$, while in Proposition \ref{perrin} this pairing is with $t\tilde \eta_{f^*}$.  The lemma follows from the commutativity of the diagram \eqref{rig-iso}.
\end{proof}

We are finally in position to provide the {\em proof of Theorem \ref{unteo}:} After specializing Corollary \ref{qualsevol} at $s=1$ we obtain
\begin{equation*}
\mathcal L_{f,1}^-(\res_p \,\kappa_{f}(1)^-) \equiv -i N \frac{\mathfrak g(\theta \chi_1 \bar \chi_2)}{\mathfrak g(\chi_1 \bar \chi_2)} \zeta_p(-1) \cdot L_p(f^*,\bar \chi_1 \chi_2,1) \cdot \mathcal L_{\theta,1}(\res_p \, \kappa_{\theta}(1)) \pmod{\mathfrak p^t}
\end{equation*}

Recall that Proposition \ref{Leop} and Theorem \ref{kato-law} assert that \[
\kappa_{\theta}(1) = (1- \theta (p)) \cdot c_{\theta}, \quad \kappa_{f}(1) = \mathcal E_f \cdot \kappa_{f}
\]
and hence
\begin{equation*}
\mathcal E_f \mathcal L_{f,1}^-( \res_p\, \kappa_{f}^-) \equiv -i N \frac{\mathfrak g(\theta \chi_1 \bar \chi_2)}{\mathfrak g(\chi_1 \bar \chi_2)} \zeta_p(-1)  (1- \theta (p)) \cdot L_p(f^*,\bar \chi_1 \chi_2,1) \cdot \mathcal L_{\theta,1}(\res_p \,c_{\theta}) \pmod{\mathfrak p^t}
\end{equation*}

Recall we have set
$$
 \ell = -iN \cdot \frac{\mathfrak g(\theta \chi_1 \bar \chi_2)}{\mathfrak g(\theta) \mathfrak g(\chi_1 \bar \chi_2)} \cdot (1-\theta(p)) \cdot \zeta_p(-1) \cdot L_p(f^*,\bar \chi_1 \chi_2,1).
$$
Using Lemma \ref{pr-iguals}  together with Lemma \ref{l-theta-iso}, we deduce the  equality of local classes
\begin{equation}\label{quasi}
\mathcal E_f \cdot  \res_p \, \kappa_f^- \equiv \ell \cdot \res_p\, c_{\theta} \, \pmod{\mathfrak p^t}
\end{equation}
in $H^1(\mathbb Q_p, \mathcal O/\mathfrak p^t(\theta)(1))$. Observe that $\res_p(\kappa_{f})^-$ is the local class obtained in cohomology by push-forward under the map induced by the projection $\bar T_f \rightarrow \bar T_{f,\circ}^{\quo}$ of \eqref{filtracio}, as already introduced in Theorem \ref{kato-law}. This corresponds, modulo $\mathfrak p^t$, to what we have called $\bar{\kappa}_{f,1}$. The first (local) part of Theorem \ref{unteo} follows.

\vspace{0.3cm}

In order to state more precisely and prove the second (global) assertion of Theorem \ref{unteo}, let us introduce some notation. Let $k=\mathbb Q(\mu_N)^+ = \Q(\zeta_N + \zeta_N^{-1})$ denote the maximal totally real subfield of $\mathbb Q(\mu_N)$ and set $d_N=[k:\Q]$. Let $\mathrm{Cl}(k)$ denote its class group. As in \eqref{units-iso} let $\mathrm{Cl}(k)[\theta]$ denote its $\theta$-eigencomponent.  It follows\footnote{ \cite[Th\'eor\`{e}me I2]{Gr} applies because Leopoldt's conjecture is known for $(k,p)$ by the work of Brumer, primes in $k$ above $p$ are totally ramified in $k(\mu_p)$ and therefore the $\omega$-component of the $\Gal(k(\mu_p)/k)$-submodule of $\mathrm{Cl}(k(\mu_p))$ generated by ideals above $p$ is trivial. \cite[Th\'eor\`{e}me I2]{Gr} thus asserts that $\mathrm{rank}_{\Z/p\Z} \,\mathrm{Cl}(k(\mu_p))[\bar\theta \omega] \otimes \Z/p\Z$ is equal to the rank of the $\bar\theta$-component of the $p$-torsion of the Galois group $\Gal(H_p/k)$ of the maximal $p$-abelian extension of $k$ unramified away from $p$.  Hence $(\mathrm{Gr})$ follows because the Hilbert class field $H/k$ is contained in $H_p$ and $\Gal(H/k)=\mathrm{Cl}(k)$.} from the work of G. Gras \cite[Th\'eor\`{e}me I2]{Gr} that
\begin{equation}\label{Gras-hyp}
\mathrm{rank}_{\Z/p\Z} \,\mathrm{Cl}(k)[\bar\theta] \otimes \Z/p\Z \leq  \mathrm{rank}_{\Z/p\Z} \,\mathrm{Cl}(k(\mu_p))[\bar\theta \omega] \otimes \Z/p\Z,
\end{equation}
where $\omega: (\Z/p\Z)^\times \ra \bar\Q^\times$ is the Teichm\"uller character. This inequality may be regarded as an instance of Leopoldt's spiegelungssatz.

Next lemma is conditional on the assumption, that we call $\bar\theta$-regularity, that \eqref{Gras-hyp} is an equality.

To place in context this hypothesis, let $R_p(k)$ denote the $p$-adic regulator  of $k$. As explained in e.g.\,\cite[Def. 2.3]{Gr-Reg}, one always has $\ord_p R_p(k) \geq d_N-1$. It is shown in loc.\,cit.\,that $(\mathrm{Gr})$ is an equality for all non-trivial even Dirichlet characters of conductor $N$ if and only if $\ord_p R_p(k) = d_N-1$. We refer to \cite[\S 7.3]{Gr-Reg} for conjectures predicting that such an equality is expected to hold for all primes $p$ away from a set of density $0$.

\begin{lemma}\label{global-local}
Assuming $p$ is {\em $\bar\theta$-regular}, the global-to-local restriction map \[ H_{\fin}^1(\mathbb Q, \mathcal O_{\mathfrak p}(\theta)(1)) \rightarrow H^1(\mathbb Q_p,\mathcal O_{\mathfrak p}(\theta)(1)) \] is an isomorphism.
\end{lemma}
\begin{proof}
Recall from the proof of Lemma \ref{l-theta-iso} the definition of the group $U_p(N)$ of local units. Consider the following commutative diagram, where vertical arrows are isomorphisms induced from Kummer theory and the upper horizontal arrow stands for the map corresponding to localization at $p$:
\[\xymatrix{
		H_{\fin}^1(\mathbb Q, \mathcal O_{\mathfrak p}(\theta)(1)) \ar[r]& H^1(\mathbb Q_p,\mathcal O_{\mathfrak p}(\theta)(1)) \\
		\mathbb Z[\mu_N]^{\times}[\bar \theta] \otimes \mathcal O_{\mathfrak p} \ar[r]\ar[u]& U_p(N)[\bar \theta] \otimes \mathcal O_{\mathfrak p} \ar[u]
	}  \]
The bottom horizontal arrow is injective because it is induced by the natural inclusion $\mathbb Z[\mu_N]^{\times} \hookrightarrow U_p(N)$. Moreover, since $\theta$ is even and nontrivial, both $\mathbb Z[\mu_N]^{\times}[\bar \theta] \otimes \mathcal O_{\mathfrak p}$ and $U_p(N)[\bar \theta]$ are $\mathcal O_{\mathfrak p}$-modules of rank $1$. The cokernel $$Q[\bar\theta]=U_p(N)[\bar \theta] \otimes \mathcal O_{\mathfrak p}/\mathbb Z[\mu_N]^{\times}[\bar \theta] \otimes \mathcal O_{\mathfrak p}$$ is thus a finite group.

In order to prove the lemma it thus suffices to show that $Q[\bar\theta]$ is trivial. Write $k=\mathbb Q(\mu_N)^+$ (resp. $\mathbb Z[\mu_N]^+$) for the maximal totally real subfield of $\mathbb Q(\mu_N$) (resp. its ring of integers). Let $U_p^1(N) = \prod_{v|p} (\mathbb Z[\mu_N]_v^+)^1$, where $(\mathbb Z[\mu_N]_v^+)^1$ stands for the set of local units in $\mathbb Z_p[\mu_N]_v^+$ which are congruent to 1 modulo $v$. Let $U^+(N)$ denote the closure of the set of units of $\mathbb Z[\mu_N]^+$ congruent to 1 modulo each place above $p$, diagonally embedded in $U_p^1(N)$. Note that $Q[\bar\theta]=U_p^1(N)/U^+(N)$.

According to \cite[Chapter 4, Theorem 7.8]{Neu}, $Q[\bar\theta] \simeq \Gal(H_p/H)[\bar\theta]$, where $H_p$ (resp. $H$) is the maximal $p$-abelian extension of $k$ unramified away from primes above $p$ (resp. everywhere unramified). Here the $\theta$-eigencomponent on the Galois group is taken with respect to the natural action of $\Gal(k/\Q)$ by conjugation on $\Gal(H_p/H)$. The lemma hence follows from the running hypothesis.
\end{proof}

Assuming {\em $\bar\theta$-regularity}, Lemma \ref{global-local} allows us to upgrade \eqref{quasi} to an equality of global classes in $H^1(\mathbb Q, \mathcal O/\mathfrak p^t(\theta)(1))$, namely
\[ \mathcal E_f \cdot  \kappa_{f,1} \equiv \ell \cdot  c_{\theta} \pmod{\mathfrak p^t}. \] Theorem \ref{unteo} follows.

\section{Second congruence relation}\label{second-section}

As in the introduction, let $N> 1$ be a positive integer and $\theta: (\Z/N\Z)^\times \ra \bar\Q^\times$ an even Dirichlet character.  Let $f \in S_2(N,\theta)$ a normalized cuspidal eigenform of level $N$, weight $2$ and nebentype $\theta$. Fix a prime $p\nmid 6N \varphi(N)$ and assume as in \eqref{congruence} that $f \equiv E_2(\theta,1) \,\, \mathrm{mod} \, \mathfrak p^t$ for some $t \geq 1$. This implies that $L_p(\bar \theta,-1) \equiv 0 \pmod{\fp^t}$.

We keep the notations introduced along \S \ref{sec:back} and \S \ref{first-section}. Recall that the value of the modular unit $u_{\chi_1,\chi_2}$ at $\infty$ is some power of the circular unit $c_{\chi_1}$, and likewise for $u_{\xi_1,\xi_2}$. For the sake of concreteness, in the statement below we normalize them so that $u_{\chi_1,\chi_2}(\infty) = c_{\chi_1}$ and  $u_{\xi_1,\xi_2}(\infty) = c_{\xi_1}$ although any other normalization would work.

\begin{propo}\label{lift}
The Beilinson--Kato class $\kappa_f \in H^1(\Q, T_{f,Y}(1))$ may be lifted to a class in $H^1(\Q, T_{f,X}(1))$ if and only if $\bar \kappa_{f,1} = 0$.
\end{propo}

\begin{proof}
Recall from Proposition \ref{ses-y} the short exact sequence of $G_{\Q}$-modules
\begin{equation}\label{ret-xyc}
0 \longrightarrow T_{f,X} \longrightarrow T_{f,Y} \longrightarrow \mathcal O/\mathfrak p^t(\theta) \longrightarrow 0.
\end{equation}
After taking a Tate twist, it gives rise to the long exact sequence in cohomology
\begin{equation}\label{xyc}
0 \longrightarrow H^1(\mathbb Z[1/Np], T_{f,X}(1)) \longrightarrow H^1(\mathbb Z[1/Np], T_{f,Y}(1)) \longrightarrow H^1(\mathbb Z[1/Np], \mathcal O/\mathfrak p^t(\theta)(1))
\end{equation}
The last map sends $\kappa_f$ to $\bar\kappa_{f,1}$. Hence the latter vanishes if and only if $\kappa_f$ belongs to $H^1(\mathbb Z[1/Np], T_{f,X}(1))$.
\end{proof}

Assume for the remainder of this section that $\bar \kappa_{f,1} = 0$ and hence $\kappa_f \in H^1(\Q, T_{f,Y}(1))$ lifts to a class in $H^1(\Q, T_{f,X}(1))$, that by a slight abuse of notation we continue to denote with the same symbol.


As in the introduction, we may define the global class
$$
\bar \kappa_{f,2} = \bar\pi_{2*}(\bar\kappa_{f}) \in H^1(\mathbb Q, \mathcal O/\mathfrak p^t(2)).
$$

\begin{theorem}\label{dosteo} (Second congruence relation)
Assume $\theta$ is primitive and $L_p'(\bar \theta,-1) \not\equiv 0 \pmod{\fp}$. Then the following equality holds in
$H^1(\mathbb Q, \mathcal O/\mathfrak p^t(2))$:
\[ \bar{\kappa}_{f,2} = \frac{L_p'(\bar \theta,-1)}{1-p^{-1}} \cdot  \frac{\bar c_{\xi_1} \cup \bar c_{\chi_1}}{\cup \log_p(\varepsilon_{\cyc})} \pmod{\mathfrak p^t}. \]
Here, $1/\cup  \log_p(\varepsilon_{\cyc})$ denotes the inverse of the map \[ H^1(\mathbb Q, \mathcal O/\mathfrak p^t(2)) \rightarrow H^2(\mathbb Q, \mathcal O/\mathfrak p^t(2)), \quad \kappa \mapsto \kappa \cup \log_p(\varepsilon_{\cyc}), \] which is invertible under our running assumptions.

\end{theorem}

\subsection{Cohomology and Eisenstein quotients}

For any $r \geq 1$ and $j\in \Z$
let $$H_r(j) = H_{\et}^1(\bar{X}_1(Np^r),\mathcal O_{\mathfrak p}(j))^{\ord}$$ denote the ordinary component of the \'etale cohomology group $H_{\et}^1(\bar{X}_1(Np^r),\mathcal O_{\mathfrak p}(j))$ with respect to the Hecke operator $U_p$.
This is naturally an $\cO_{\fp}[G_{\Q}]$-module and we may simply denote it $H_r$ when the Galois action is understood or irrelevant. Let $\mathfrak h_r$ be the subring of $\End_{\mathcal O_{\mathfrak p}}(H_r)$ spanned over $\mathcal O_{\mathfrak p}$ by the Hecke operators $T_n$, $(n,N)=1$. The Eisenstein ideal $I_r = I_{\Eis,r} \subset \mathfrak h_r$ is the $\cO_{\fp}$-submodule of $\mathfrak h_r$ generated by $U_p-1$ and $T_{\ell}-\ell \langle \ell \rangle-1$, for primes $\ell \nmid Np$; here, $\langle \ell \rangle$ stands for the usual diamond operator.

Passing to the projective limit we may define:
\[ H(j) := H_{\et}^1(\bar{X}_1(Np^{\infty}),\mathcal O_{\mathfrak p}(j))^\ord =  \lim_{\leftarrow} H_r(j), \quad \mathfrak h = \lim_{\leftarrow} \mathfrak h_r, \quad I = \lim_{\leftarrow} I_{r} \subset \mathfrak h. \]

 The ideal $I$ is a height one ideal contained in the maximal ideal $\fM=(I,\fp)$; for any $t\geq 1$ we shall denote $\fM^{(t)}=(I,\fp^t)$, so that $\fM=\fM^{(1)}$. The ideal $I$ is the intersection of a finite number of height one prime ideals $P  \subset \fM$, each of which corresponds to a weight two eigenform that is congruent to an Eisenstein series mod $\fp$, like the  modular form $f$ of the introduction.

Let $\Lambda_N := \lim_{\leftarrow} \mathcal O_{\mathfrak p}[(\mathbb Z/Np^r \mathbb Z)^{\times}]$ denote the Iwasawa algebra of tame level $N$. Any Dirichlet character $\psi: (\Z/Np\Z)^\times \ra \cO_{\mathfrak{p}}^\times$ may be extended  by linearity to yield a homomorphism
\[ \psi: \Lambda_N \longrightarrow \mathcal O_{\mathfrak p}[(\mathbb Z/Np\mathbb Z)^{\times}] \lra  \mathcal O_{\mathfrak p} \]
that we continue to denote with the same symbol.

For any $\Lambda_N$-module $M$ let $M_{\psi} = M \otimes_{\Lambda_N, \psi} \cO_{\mathfrak{p}}$ stand for the associated $\psi$-isotypical component. Note that $\Lambda_N = \oplus \Lambda_{N,\psi}$ where $\psi$ ranges over all characters of $(\Z/Np\Z)^\times$ and $\Lambda_{N,\psi} \simeq \mathcal O_{\mathfrak p}[[1+p\Z_p]]$.




We begin by rephrasing the results of \cite[\S9]{FK} on Sharifi's conjecture in a convenient way for our purposes. As we have seen at the beginning of \S \ref{first-section}, and more precisely in the discussion before equation \eqref{HS-iso}, the Hochschild--Serre spectral sequence in \'etale cohomology yields the commutative diagram of $\Lambda_{N,\theta}[G_\Q]$-modules:

\begin{equation}\label{diag1}
\xymatrix{
		H_{\et}^2(X_1(Np^{\infty}),\mathcal O_{\mathfrak p}(2))^{\ord}_{\theta} \ar[r]\ar[d] & H_{\et}^2(X_1(Np),\mathcal O_{\mathfrak p}(2))^{\ord}_{\theta} \ar[d] \\
		H^1(\mathbb Q,H(2)_{\theta}) \ar[r] & H^1(\mathbb Q, H_{\et}^1(\bar{X}_1(Np),\mathcal O_{\mathfrak p}(2))_{\theta}^{\ord}).
	}
\end{equation}

The module $H(2)$ is endowed with an action of complex conjugation, yielding the decomposition $H(2) = H(2)^+ \oplus H(2)^-$; in the sequel we shall employ a similar notation for any $\cO_{\fp}$-module acted on by complex conjugation. It follows from \cite[Prop. 6.3.2]{FK} that the quotient
\begin{equation}\label{H+}
(H(2)/\fM^{(t)})_\theta^+
\end{equation}  is still endowed with a compatible action of $G_{\Q}$.

Our running assumptions imply that the (mod $\mathfrak{p}^t$) Galois representation $\bar{T}_{f,X}(1) = T_{f,X}(1) \otimes \mathcal O/\mathfrak p^t$ arises as a quotient of $H(2)/\fM^{(t)}$ as $\mathcal O/\mathfrak p^t[G_{\Q}]$-modules. Since the nebentype of $f$ is $\theta$, it belongs to the $\theta$-isotypical component of the latter. Denote $$\pi_f: (H(2)/\fM^{(t)})_\theta^+ \lra \bar{T}_{f,X}(1)^+$$ the resulting projection.

Recall also from \eqref{filtracio} that there is a short exact sequence for $\bar T_{f,X}$, which according to \eqref{ses-fk0} and \eqref{isos} is compatible with the action of complex conjugation.
In particular, $(\bar{T}_{f,X}^-)(1) = (\bar{T}_{f,X}(1))^+ \simeq \mathcal O/\mathfrak p^t(2)$ and this is naturally a quotient of \eqref{H+}. Henceforth we fix the canonical isomorphism provided by \cite[6.3.18 and 7.1.11]{FK} in order to identify
\begin{equation}\label{rig-iso2}
\bar{T}_{f,X}(1)^+ = \mathcal O/\mathfrak p^t(2)
\end{equation}
as $\mathcal O/\mathfrak p^t[G_{\Q}]$-modules.

Summing up there is a  commutative diagram of $G_\Q$-modules, where the horizontal arrows arise from specializing to $r=1$, i.e.\,level $Np$:

\begin{equation}\label{diag2}
\xymatrix{
		H^1(\mathbb Q,H(2)_{\theta}) \ar[r]\ar[d] & H^1(\mathbb Q, H_{\et}^1(\bar{X}_1(Np),\mathcal O_{\mathfrak p}(2))_{\theta}^{\ord}) \ar[d] \\
		H^1(\mathbb Q,(H(2)/\fM^{(t)})_{\theta}^+) \ar[r]\ar[dr] & H^1(\mathbb Q, (H_{\et}^1(\bar{X}_1(Np),\mathcal O_{\mathfrak p}(2))^{\ord}/(I_1,\fp^t))_{\theta}^+) \ar[d] \\
		 &  H^1(\mathbb Q,\bar{T}_{f,X}(1)^+) = H^1(\mathbb Q,\mathcal O/\mathfrak p^t(2))
	}
\end{equation}

\subsection{Fukaya--Kato maps}

Define the module $\mathcal Q$ as in \cite[\S6.3.1]{FK}, with a twist by the square of the cyclotomic character, namely $$\mathcal Q := (H(2)/I H(2))_{\theta}^+.$$

Recall from \S \ref{first-section} (cf.\,e.g.\,Prop.\,\ref{Leop}) the Kubota--Leopoldt $p$-adic $L$-function $L_p(\bar \theta) \in \Lambda_{N,\theta}$ attached to the Dirichlet character $\bar \theta$.
As shown in \cite[\S6.1.7]{FK}, there is an isomorphism of Galois modules
$$(\mathfrak h/I)_{\theta} \simeq \Lambda_{N,\theta}/(L_p(\bar \theta))(\underline{\varepsilon}_{\cyc}),$$ where $\underline{\varepsilon}_{\cyc}$ is the $\Lambda$-adic cyclotomic character introduced in \eqref{lambda-eps-cyc} and the identification follows from \cite[\S6.3]{FK}, with the conventions recalled in \S9.1.2 and 9.1.4.
Moreover, and as a consequence of the proof of the Iwasawa main conjecture by Mazur and Wiles, in \S6.3.18 of loc.\,cit. the authors show that there are isomorphisms of Galois modules
\begin{equation}\label{isosMW}
\mathcal Q \xrightarrow{\simeq} \Lambda_{N,\theta}/(L_p(\bar \theta))(\underline{\varepsilon}_{\cyc})(2) \simeq (\mathfrak h(2)/I \mathfrak h(2))_{\theta}.
\end{equation}

Recall that $H^i(\Z[1/Np],M) \subset H^i(\mathbb Q,M)$ stands for the set of classes which are unramified at primes dividing $Np$. Shapiro's lemma gives an isomorphism
\[
 \lim_{\leftarrow} H^2(\Z[1/Np,\zeta_{Np^r}],\mathcal O_{\mathfrak p}(2)) \simeq H^2(\Z[1/Np], \Lambda_N(\underline{\varepsilon}_{\cyc})(2)).
 \]
As a piece of notation, and following the definition of \cite[\S5.2.6]{FK}, set \[ \mathcal S = \lim_{\leftarrow} H^2(\Z[1/Np, \zeta_{Np^r}],\mathcal O_{\mathfrak p}(2))^+  \simeq H^2(\Z[1/Np], \Lambda_N(\underline{\varepsilon}_{\cyc})(2))^+. \]

In \cite[\S9.1]{FK} Fukaya and Kato established the existence of isomorphisms
\[
 \FK_1 : H^1(\Z[1/Np], \mathcal Q) \simeq \mathcal S_{\theta}, \quad \FK_2: \mathcal S_{\theta} \simeq H^2(\Z[1/Np], \mathcal Q)
 \]
arising from  the long exact sequence in cohomology induced by the short exact sequence
\begin{equation*}
0 \rightarrow \Lambda_{N,\theta}(\underline{\varepsilon}_{\cyc})(2) \xrightarrow{\cdot L_p(\bar \theta)} \Lambda_{N,\theta}(\underline{\varepsilon}_{\cyc})(2) \rightarrow \mathcal Q \rightarrow 0.
\end{equation*}
stemming from \eqref{isosMW}.

In particular, the map we have denoted as $\FK_2$ is just  the $+$-component of the homomorphism
\begin{equation}\label{FK2-exp}
H^2(\Z[1/Np], \Lambda_{N,\theta}(\underline{\varepsilon}_{\cyc}) (2)) \longrightarrow H^2(\Z[1/Np], \Lambda_{N,\theta}/(L_p(\bar \theta))(\underline{\varepsilon}_{\cyc})(2))
\end{equation}
induced by (a twist of) the natural projection $\Lambda_{N,\theta} \lra \Lambda_{N,\theta}/L_p(\bar \theta)$.  For an explicit description of $\FK_1$, see \cite[\S9.1]{FK}.

The main result of \S9.2 of loc.\,cit. asserts that the map
\begin{equation}\label{ev-inf}
\mathrm{ev}_\infty:  H_{\et}^2(X_1(Np^\infty), \mathcal O_{\mathfrak p}(2))_{\theta} \longrightarrow \lim_{\leftarrow} H^2(\Z[1/Np,\zeta_{Np^r}],\mathcal O_{\mathfrak p}(2))_{\theta}
\end{equation}
induced by evaluation at the cusp $\infty$ factors through the Eisenstein quotient, as stated below.

\begin{propo}[Fukaya--Kato]\label{aux-0}
The map $\mathrm{ev}_\infty$ of \eqref{ev-inf} agrees with  the composition \[ H_{\et}^2(X_1(Np^\infty), \mathcal O_{\mathfrak p}(2))_{\theta}  \rightarrow H^1(\Z[1/Np],H(2)_{\theta}) \rightarrow H^1(\Z[1/Np], \mathcal Q) \simeq \mathcal S_{\theta}\]
where:
\begin{itemize}
\item the first map is the composition of $H^2(X_1(Np^\infty), \mathcal O_{\mathfrak p}(2)) \ra H^2(X_1(Np^\infty), \mathcal O_{\mathfrak p}(2))^\ord_\theta$ and the left vertical arrow in \eqref{diag1}, both restricted to the subspace of classes unramified at the primes dividing $Np$;

\item the second map is induced by the projection $H(2)_{\theta} \rightarrow \mathcal Q$;

\item the last isomorphism is $\FK_1$.
\end{itemize}
\end{propo}

In \cite[\S9.3]{FK} Fukaya and Kato further introduced two distinguished morphisms
\begin{eqnarray}
 a,b: H^1(\Z[1/Np], \mathcal Q) \rightarrow H^2(\Z[1/Np], \mathcal Q) \\ \nonumber
 a=  \FK_2 \circ \FK_1, \\ \nonumber
 b= \cup \, (1-p^{-1})\log_p(\varepsilon_{\cyc})
 \end{eqnarray}
where $\varepsilon_{\cyc} \in H^1(\mathbb Q, \mathcal O^\times_{\mathfrak p})$ stands for the cyclotomic character. Note that $(1-p^{-1})\log_p$ takes values in $\Z_p$ and hence $b$ is indeed well-defined. Under these conditions, they show the following.

\begin{propo}[Fukaya--Kato]\label{aux-1} Let $L_p'(\bar \theta) \in \Lambda_{N,\theta}$ denote the derivative of $L_p(\bar \theta)$.
Then
\begin{equation}\label{b-a}
b = L_p'(\bar \theta) \cdot a.
\end{equation}
\end{propo}
\begin{proof}
This follows from \cite[Proposition 9.3.1]{FK}; in particular, we just need the restriction of $\theta$ to $(\mathbb Z/N\mathbb Z)^{\times}$ being primitive, since Lemma 9.1.3 of loc.\,cit.\,also works in this setting.
\end{proof}

\begin{coro}\label{invertible}
Assume $L_p'(\bar \theta,-1)$ is a $p$-adic unit. Then the map \[ H^1(\Z[1/Np], \mathcal O/\mathfrak p^t(2)) \longrightarrow H^2(\Z[1/Np],\mathcal O/\mathfrak p^t(2)), \quad \kappa \mapsto \kappa \cup (1-p^{-1}) \log_p(\varepsilon_{\cyc}) \] is invertible.
\end{coro}
\begin{proof}
Observe that $\FK_1$ and $\FK_2$ are $\Lambda$-adic isomorphisms, as it has been proved in \cite[\S 9.1]{FK}. Hence, once we consider the specialization map at the trivial character, we still have isomorphisms of $\mathcal O_{\mathfrak p}$-modules. The same must be true for their composition multiplied by the $p$-adic unit $L_p'(\bar \theta,-1)$, and according to Proposition \ref{aux-1} and the definitions provided in \cite[\S 4.1.3]{FK}, this is precisely the above map.
\end{proof}

After applying the Fukaya--Kato map $\FK_1$ to the bottom row of diagram \eqref{diag2}, restricting to the subspace of unramified classes at primes dividing $Np$, we reach the commutative diagram
\begin{equation}\label{diag3}
\xymatrix{
		H^1(\Z[1/Np],(H(2)/\fM^{(t)})_{\theta}^+) \ar[r]\ar^{\bar{\FK}_1}[d] &
		H^1(\Z[1/Np],\bar{T}_{f,X}(1)^+) = H^1(\Z[1/Np],\mathcal O/\mathfrak p^t(2)) \ar^{\bar{\FK}_1(r=1)}[d] \\
		H^2(\Z[1/Np],\Lambda_{N,\theta}/\fp^t(\underline{\varepsilon}_{\cyc})(2)) \ar[r] & H^2(\Z[1/Np], \mathcal O/{\mathfrak p^t}(2)),
	}
\end{equation}
where the left-most vertical map is $\bar{\FK}_1 = \FK_1 \pmod{\mathfrak p^t}$. As in \eqref{diag2}, the horizontal arrows are specialization at $r=1$, and the right-most vertical arrow is accordingly  the specialization of $\bar{\FK}_1$ at $r=1$.

We may further apply now Fukaya--Kato's map $\bar{\FK}_2 = \FK_2  \pmod{\mathfrak p^t}$ to the above diagram and obtain the following one:
\begin{equation}\label{diag4}
\xymatrix{
		H^2(\Z[1/Np],\Lambda_{N,\theta}/\fp^t(\underline{\varepsilon}_{\cyc})(2)) \ar[r]\ar[d]^{\bar{\FK}_2} & H^2(\Z[1/Np], \mathcal O/\mathfrak p^t (2)) \ar[d]^{||} \\
		H^2(\Z[1/Np],(H(2)/\fM^{(t)})_{\theta}^+)\ar[r] & H^2(\Z[1/Np], \mathcal O/\mathfrak p^t (2))
	}
\end{equation}

Again the horizontal maps are specialization in level $Np$ at $r=1$ and the right-most vertical map is the specialization of $\bar{\FK}_2$ at $r=1$. In view of \eqref{FK2-exp} the latter may be identified with the identity map: according to the definitions provided in \cite[\S 6.1.6, 4.1.3]{FK} and with our current conventions, the specialization of $L_p(\bar\theta)$ at $r=1$ is  $$L_p(\bar \theta,-1) = (1-\bar \theta(p)p) \cdot L(\bar \theta,-1) = -(1-\bar \theta(p)p) \cdot \frac{B_{2}(\bar \theta)}{2},$$ which vanishes$\pmod{\mathfrak p^t}$ in light of our  assumptions.

\subsection{Proof of Theorem \ref{dosteo}}

We can finally prove Theorem \ref{dosteo}. With a slight abuse of notation, we identify global units with their image in cohomology under the Kummer map.

According to Proposition \ref{aux-0}, we have
\begin{equation}\label{first-step}
\FK_1(\bar{\kappa}_{f,2}) = \mathrm{ev}_\infty(u_{\xi_1,\xi_2} \cup u_{\chi_1,\chi_2}) = \bar c_{\xi_1} \cup \bar c_{\chi_1}  \pmod{\mathfrak p^t}, \end{equation}
where the circular units involved in the cup product are those resulting from the evaluation at infinity of the modular units $u_{\xi_1,\xi_2}$ and $u_{\chi_1,\chi_2}$, respectively. Recall we are assuming that $\xi_1 = \bar \chi_1$, and it was proved in Theorem \ref{kato-law} that $\kappa_f$ is unramified everywhere.

Next, we apply $\FK_2$ to both sides of \eqref{first-step}. Proposition \ref{aux-1} together with the commutativity of \eqref{diag4} allow us to establish that \[ \bar \kappa_{f,2} \cup (1-p^{-1}) \log_p(\varepsilon_{\cyc}) = L_p'(\bar \theta,-1) \cdot (\bar c_{\xi_1} \cup \bar c_{\chi_1}). \]

Theorem \ref{dosteo} finally follows from Corollary \ref{invertible}.


\begin{thebibliography}{RR3}

\bibitem[Bei]{Be}
A.~Beilinson, {\em Higher regulators and values of $L$-functions}, {\em Current problems in mathematics} {\bf 24} (1984), 181--238.

\bibitem[Bel]{Bel}
J.~Bella\"iche, An introduction to the conjecture of Bloch and Kato, available at \url{http://people.brandeis.edu/~jbellaic}.


\bibitem[BD1]{BD-Joch}
M.~Bertolini and H.~Darmon.
{\em Euler systems and Jochnowitz congruences}, {\em Amer. J. of Math.} {\bf 121} (1999), no. 2, 259--281.

\bibitem[BD2]{BD}
M.~Bertolini and H.~Darmon.
{\em Kato's Euler system and rational points on elliptic curves I: a $p$-adic Beilinson formula}, {\em Israel Journal of Mathematics} {\bf 199} (2014), 163--178.

\bibitem[BDR1]{BDR2}
M.~Bertolini, H.~Darmon, V.~Rotger, {\em Beilinson--Flach elements and Euler systems II: $p$-adic families and the Birch and Swinnerton-Dyer conjecture}, {\em J. Algebraic Geometry} {\bf 24} (2015), 569--604.

\bibitem[Bes]{Bes}
A.~Besser, {\em Syntomic regulators and $p$-adic integration I: rigid syntomic regulators}, {\em Israel J. Math.} {\bf 120} (2000), 291--334.

\bibitem[BK]{BK}
S.~Bloch and K.~Kato. {\em $L$-functions and Tamagawa numbers of motives}, {\em The Grothendieck Festschrift I, Progr. Math.} {\bf 108}, 333--400 (1993), Birkhauser.

\bibitem[Bru]{Bru}
A.~Brumer, {\em On the units of algebraic number fields}, {\em Mathematika} {\bf 14} (1967), 121--124.

\bibitem[Co]{Col}
R.~Coleman, {\em Division values in local fields}, {\em Invent. Math.} {\bf 53} (1979).

\bibitem[Con]{Con}
K.~Conrad, Infinite series in $p$-adic fields, available at \url{https://kconrad.math.uconn.edu/blurbs/gradnumthy/infseriespadic.pdf}.


\bibitem[DS]{DS}
F.~Diamond and J.~Shurman. {\em A First Course in Modular Forms}. {\em Springer}, New York, (2005).



\bibitem[Eu]{Tale}
M.~Bertolini, F.~Castella, H.~Darmon, S.~Dasgupta, K.~Prasanna, and V.~Rotger. {\em $p$-adic $L$-functions and Euler systems: a tale in two trilogies}, in {Automorphic forms and Galois representations}, vol.~1, LMS Lecture Notes {\bf 414}, Cambridge University Press (2014) 52--102.

\bibitem[FK]{FK}
T.~Fukaya and K.~Kato. {\em On conjectures of Sharifi}. Preprint.

\bibitem[FKS]{FKS}
T.~Fukaya, K.~Kato, and R.~Sharifi. {\em Modular symbols in Iwasawa theory}, {\em Iwasawa theory 2012. Contributions in Mathematical and Computational Sciences} (2014).

\bibitem[Gr]{Gr}
G.~Gras. {\em Groupe de Galois de la $p$-extension ab\'elienne $p$-ramifi\'ee maximale d'un corps de nombres}, {\em J. Reine Angew. Math.} {\bf 333} (1982), 86--133.

\bibitem[Gr2]{Gr-Reg}
G.~Gras. {\em Les $\theta$-r\'egulateurs locaux d'un nombre alg\'ebrique-- conjectures $p$-adiques}, {\em Canadian J. Math.} {\bf 68} (2016), 571--624.

\bibitem[GV]{GV}
R.~Greenberg and V.~Vatsal. {\em Iwasawa invariants of elliptic curves}, {\em Invent. Math.} {\bf 142} (2000), 17--63.

\bibitem[Han]{Han}
D.~Hansen, {\em Iwasawa theory of overconvergent modular forms, I: Critical-slope p-adic L-functions}. Preprint (2016).

\bibitem[HV]{HV}
J.~Heumann and V.~Vatsal {\em Modular symbols, Eisenstein series, and congruences}, {\em J. de Th\'eorie des Nombres de Bordeaux} {\bf 26} (2014), no. 3, 709--756.

\bibitem[Hi]{Hi}
H.~Hida, {\em A $p$-adic measure attached to the zeta functions associated with two elliptic modular forms II}, {\em Ann. Inst. Fourier (Grenoble)} {\bf 38} (1988), no. 3, 1--83.

\bibitem[Ka]{Kato}
K.~Kato. {\em $p$-adic Hodge theory and values of zeta functions of modular forms}, {\em Ast\'erisque} {\bf 295} (2004), no. 9, 117--290.


\bibitem[KLZ]{KLZ}
G.~Kings, D.~Loeffler, and S.L.~Zerbes. {\em Rankin-Eisenstein classes and explicit reciprocity laws}, {\em Cambridge J. Math} {\bf 5} (2017), no. 1, 1--122.




\bibitem[LZ]{LZ}
D.~Loeffler and S.L.~Zerbes. {\em Iwasawa theory and $p$-adic $L$-functions over $\mathbb Z_p^2$-extensions}, {\em Int. J. Number Theory} {\bf 10} (2014), no. 8, 2045--2096.

\bibitem[LPSZ]{LPSZ}
D.~Loeffler, V.~Pilloni, C.~Skinner, S.L.~Zerbes. {\em Higher Hida theory and p-adic L-functions for $\mathrm{GSp}(4)$}, to appear in {\em Duke Math. J.}

\bibitem[LSZ]{LSZ} 
D.~Loeffler, C.~Skinner, S.L.~Zerbes. {\em An Euler system for $\mathrm{GU}(2, 1)$}, preprint.

\bibitem[Maz]{Maz}
B.~Mazur, {\em On the Arithmetic of Special Values of L Functions}, Inv. Math. {\bf 55} (1979), 207--240.


\bibitem[Mil]{Milne}
J.~Milne, \'Etale cohomology.

\bibitem[Nek]{Nek}
J.~Nekov\'ar. {\em $p$-adic Abel-Jacobi maps and $p$-adic heights}. Lecture notes of the author's lecture at the 1998 Conference ``The Arithmetic and Geometry of Algebraic Cycles".

\bibitem[Neu]{Neu}
J.~Neukirch. {\em Class field theory}, {\em Springer}, (1969).

\bibitem[Och]{Och}
T.~Ochiai. {\em On the two-variable Iwasawa main conjecture for Hida deformations}, {\em Compositio Math.} {\bf 142} (2006), 1157--1200.

\bibitem[Oh1]{Oh1}
M.~Ohta. {\em Ordinary $p$-adic \'etale cohomology groups attached to towers of elliptic modular curves}, {\em Compositio Math.} {\bf 115} (1999), 241--301.

\bibitem[Oh2]{Oh2}
M.~Ohta. {\em Ordinary $p$-adic \'etale cohomology groups attached to towers of elliptic modular curves II}, {\em Math. Annalen} {\bf 318} (2000), no. 3, 557--583.

\bibitem[Oh3]{Oh3}
M.~Ohta. {\em Congruence modules related to Eisenstein series}, {\em Ann. Scient. \'Ec. Norm. Sup.} {\bf 36} (2003), no. 4, 225--269.

\bibitem[PR]{PR}
B.~Perrin-Riou. {\em La fonction L p-adique de Kubota--Leopoldt}, Arithmetic geometry (Tempe, AZ, 1993), 65--93, Contemp. Math. {\bf 174}, Amer. Math. Soc., Providence, RI, 1994.


\bibitem[RiRo1]{RiRo1}
O.~Rivero and V.~Rotger. {\em Derived Beilinson--Flach elements and the arithmetic of the adjoint of a modular form}, to appear in {\em J. Eur. Math. Soc}.

\bibitem[Sch]{Sch}
A. Scholl, {\em An Introduction to Kato's Euler Systems}, {\em Galois Representations in Arithmetic Algebraic Geometry}, {\em Cambridge University Press} (2010).

\bibitem[Sha]{Sh}
R.~Sharifi. {\em A reciprocity map and the two-variable p-adic L-function}, {\em Annals of Mathematics} {\bf 171} (2011), no. 1, 251--300.

\bibitem[St1]{St-book}
G.~Stevens. {\em Arithmetic on Modular Curves}, {\em Progress in Mathematics}, vol. 20. Birkhauser Boston (1982).

\bibitem[St2]{St-TAMS}
G.~Stevens. {\em The cuspidal group and special values of $L$-functions}, {\em Trans. Amer. Math. Soc}, {\bf 291} (1985), no. 2, 519--550.

\bibitem[Va]{Va}
V.~Vatsal. {\em Canonical periods and congruence formulas}, {\em Duke Math. J.}, {\bf 98} no. 2 (1999), 397--419.

\bibitem[Va2]{Va2}
V.~Vatsal. {\em Special values of anticyclotomic $L$-functions}, {\em Duke Math. J.}, {\bf 116} no. 2 (2003), 219--261.





\end{thebibliography}
\end{document}